\documentclass[12pt]{amsart}
\usepackage{amssymb,amsmath,amscd}
\newtheorem{thm}{\bf Theorem}[section]
\newtheorem{prop}[thm]{\bf Proposition}
\newtheorem{lem}[thm]{\bf Lemma}
\newtheorem{cor}[thm]{\bf Corollary}
\theoremstyle{definition}\newtheorem{exa}[thm]{\bf Example}
\theoremstyle{definition}\newtheorem{de}[thm]{\bf Definition}
\theoremstyle{definition}\newtheorem{rem}[thm]{\bf Remark}
\theoremstyle{definition}

\numberwithin{equation}{section}

\DeclareMathOperator{\p}{\mathbb P}

\usepackage{color}
\begin{document}
\title[The Heisenberg groups]{The fundamental and rigidity theorems for pseudohermitian submanifolds in the Heisenberg groups}
\author[Hung-Lin Chiu]{Hung-Lin Chiu}
\address{Department of Mathematics, National Tsing Hua University, Hsinchu, Taiwan 300, R.O.C.}
\email{hlchiu@math.nthu.edu.tw}
\thanks{Email: hlchiu@math.nthu.edu.tw}
\renewcommand{\subjclassname}{\textup{2000} Mathematics Subject Classification}

\begin{abstract} 
In this paper, we study some basic geometric properties of pseudohermitian submanifolds of the Heisenberg groups. In particular, we obtain the uniqueness and existence theorems, and some rigidity theorems. 
\end{abstract}

\maketitle

\section{Introduction}

In this paper, for $m\leq n$, we specify the ranges of indices as following
\begin{equation*}
\begin{split}
1&\leq\alpha,\beta,\gamma,\sigma,\rho,\cdots\leq n\\
1&\leq j,k,l,\cdots\leq m\\
m+1&\leq a,b,c,\cdots\leq n\\
1&\leq A,B,C,\cdots\leq 2n\\
\end{split}
\end{equation*}

\subsection{The Heisenberg groups} The origin of pseudohermitian geometry came from the construction of a pseudohermitian connection, independently by N. Tanaka \cite{T} and S. Webster \cite{We}. In this paper, the Heisenberg group is a pseudohermitian manifold and it plays the role of the model in pseudohermitian geometry. That is, any pseudohermitian manifold with vanishing curvature and torsion locally is part of the Heisenberg group.
Let $H_{n}$ be the Heisenberg group, with coordinates $(x_{\beta},y_{\beta},t)$. The group multiplication is defined by
\begin{equation*}
(x,y,t)\circ(x',y',t')=(x+x',y+y',t+t'+yx'-xy').
\end{equation*}
The associated standard CR structure $J$ and contact form $\Theta$ are defined respectively by 
\begin{equation*}
\begin{split}
J\mathring{e}_{\beta}&=\mathring{e}_{n+\beta};\ J\mathring{e}_{n+\beta}=-\mathring{e}_{\beta}\\ 
\Theta&=dt+\sum_{\beta=1}^{n}x_{\beta}dy_{\beta}-y_{\beta}dx_{\beta}, 
\end{split}
\end{equation*}
where 
\begin{equation*}
\mathring{e}_{\beta}=\frac{\partial}{\partial x_{\beta}}+y_{\beta}\frac{\partial}{\partial t},\ \  \mathring{e}_{n+\beta}=\frac{\partial}{\partial y_{\beta}}+x_{\beta}\frac{\partial}{\partial t}.
\end{equation*}
The contact bundle is $\xi=\textrm{ker} \Theta$. We refer the reader to \cite{CCG},\cite{CC} and \cite{CL} for the details about the Heisenberg groups, and to \cite{DT},\cite{Le1},\cite{Le2},\cite{T} and \cite{We} for pseudohermitian geometry.

The symmetry group $PSH(n)$ of $H_{n}$ is the group consisting of all pseudohermitian transformations. Left translations $L_{p}$ are a symmetry. Another kind of examples are a rotation $\Phi_{R}$ around the $t$-axis which is defined by
\begin{equation*}
	\Phi_{R}\left(\begin{array}{c}x\\y\\t
	\end{array}\right)=\left(\begin{array}{cc}R&0\\0&1
	\end{array}\right)\left(\begin{array}{c}x\\y\\t
	\end{array}\right)
\end{equation*}
,where $R=\left(\begin{array}{cc}A&-B\\B&A\end{array}\right)\in SO(2n)$. In \cite{CL}, we showed that each symmetry $\Phi\in PSH(n)$ has the unique decomposition  $\Phi=L_{p}\circ\Phi_{R}$, for some $p\in H_{n}$ and $R\in SO(2n)$. Since the action of $PSH(n)$ in $H_{n}$ is transitive, the Heisenberg group is a kind of Klein geometry. The corresponding Cartan geometry is just pseudohermitian geometry.

\subsection{Pseudohermitian submanifolds} We now give the definition of pseudohermitian submanifold.
\begin{de}
A $(2m+1)$-dimensional pseudohermitian manifold $(M,\hat{J},\hat{\theta})$ is called a pseudohermitian\footnote{In \cite{DT}, S. Dragomir and G. Tomassini called it {\bf isopseudo-hermitian}, instead of {\bf pseudohermitian}.} submanifold of $H_{n}$, $1\leq m\leq n$, if 
\begin{itemize}
\item $\hat{\xi}=TM\cap\xi$; 
\item $\hat{J}=J|_{\hat{\xi}}$; 
\item $\hat{\theta}=\Theta|_{M}$,
\end{itemize}
where $\hat{\xi}=\textrm{ker}\hat{\theta}$ is the contact structure on $M$.
\end{de}

\begin{exa}Suppose $M\hookrightarrow H_{n}$ is an embedded submanifold with CR dimension $n-1$. Then it is not hard to see that
 \begin{itemize}
 \item In general, $\textrm{dim}(T_{p}M\cap\xi_{p})\geq 2n-2$, for all $p\in M$.
 \item $\textrm{dim}(T_{p}M\cap\xi_{p})= 2n-2$, for generic point $p\in M$. 
 \end{itemize}
 
 All the generic points constitute the regular part of $M$, and those points $p$ such that  $\textrm{dim}(T_{p}M\cap\xi_{p})=2n-1$ are called the singular points. On the regular part $M_{re}$, assume that $T_{p}M\cap\xi$ is invariant under $J$, then it inherits a pseudohermitian structure $(\hat{J},\hat{\theta})$ from $H_{n}$ such that $(M_{re},\hat{J},\hat{\theta})$ is a pseudohermitian submanifold of $H_{n}$.
 \end{exa}
 In Section \ref{loin}, we define some local invariants for pseudohermitian submanifolds, including the second fundamental form, the normal connection and the fundamental vector field $\nu$. In addition, from Proposition \ref{fuve}, we see that the fundamental vector field $\nu$ actually describes the difference between the two Reeb vector fields $T$ and $\hat{T}$, which are, respectively,  associated with $H_{n}$ and the pseudohermitian submanifold $M$. Hence if $\nu=0$ then $\hat{T}=T$. That means that $T=\frac{\partial}{\partial t}$ is always tangent to $M$ at each point. Therefore, for such kind a submanifold, we call it {\bf vertical}\footnote{In \cite{DT}, S. Dragomir and G. Tomassini just call it {\bf pseuo-Hermitian}, instead of {\bf vertical pseudohermitian}.}. On the other hand, if $\nu\neq 0$ at each point, we call it {\bf completely non-vertical}. 
 
 \begin{exa} 
 The subspace $H_{m}=\{(z,t)\in H_{n}\ |\ z_{a}=0\}\subset H_{n}$ is a pseudohermitian submanifold of $H_{n}$. It is easy too see that $H_{m}$ is {\bf vertical}.
 \end{exa}

\begin{exa}
Let $S^{2n-1}\subset H_{n}$ be the sphere defined by  
\begin{equation*}
S^{2n-1}=\left\{(z,0)\in C^{n}\subset H_{n}\ |\ \sum_{\beta=1}^{n}z_{\beta}z_{\bar\beta}=1\right\}.
\end{equation*}
There are two pseudohermitian structures induced on $S^{2n-1}$, one is from $H_{n}$ and the other is from $C^{n}$. In subsection \ref{stspsp}, we show that these two induced pseudohermitian structures coincide. In addition, $S^{2n-1}$ is {\bf completely non-vertical}. 
\end{exa}

 \subsection{Main Theorems}
There are many literatures which were given for the problem about CR embeddability of  CR manifolds into spheres. In this paper, we obtain the fundamental theorems and rigidity theorems for pseudohermitian submanifolds in the Heisenberg groups. We have\\

{\bf Theorem A.} The induced pseudohermitian structure, the second fundamental form, the normal connection, as well as the fundamental vector field constitute a complete set of invariants for pseudohermitian submanifolds of the Heisenberg groups.\\

Theorem A is shown in Section \ref{tut}. It specifies that there are only four invariants for pseudohermitian submanifolds. That is, if two pseudohermitian submanifolds have the same such four invariants, then they are locally congruent with each other in the sense that they differ from each other nothing more than an action of a symmetry.

 Any pseudohermitian submanifold $M\subset H_{n}$ automatically satisfies a natural geometric condition which we call {\bf the integrability conditions} (defined in subsection \ref{tic}). Conversely, we will show that it is also a condition for an arbitrary pseudohermitian manifold to be (locally) embedded as a pseudohermitian submanifold of $H_{n}$.\\
  
{\bf Theorem B.} Let $(M^{2m+1},J_{M},\theta_{M})$ be a simply connected pseudohermitian manifolds satisfying the integrability conditions. Then $M$ can be embedded as a pseudohermitian submanifold of the Heisenberg groups $H_{n}$, for some $n\geq m$.\\

Theorem B is shown in Section \ref{tet}. In \cite{KO}, S.-Y. Kim and J.-W. Oh also studied the problem of characterizing pseudohermitian manifolds which are pseudohermitian embeddable into the Heisenberg groups\footnote{ In their paper, they used pseudohermitian flat sphere as ambient space, instead of the Heisenberg group. But after a Cayley transformation, this two spaces are isomorphic as pseudohermitian manifolds.}. In the case that $M$ is nondegenerate, S.-Y. Kim and J.-W. Oh used Cartan's prolongation method to show that the induced pseudohermitian structure constitutes a complete set of invariants. In addition, they gave a necessary and sufficient conditions, in terms of Webster curvature and torsion tensor, for pseudohermitian manifolds to be embeddable into the Heisenberg groups nondegenerately. This conditions is just equivalent to the integrability conditions which we define in subsection \ref{tic}. However, S.-Y. Kim and J.-W. Oh did not deal with the degenerate cases. 

In the case of CR codimension one, the nondegenerate just means that the second fundamental form does not vanish at each point. In such a case, we basically recover the results of S.-Y. Kim and J.-W. Oh. Moreover, we give the rigidity theorems for pseudohermitian degenerate submanifolds, which are shown in section \ref{trt}. We have\\ 

{\bf Theorem C.} Let $(M,\hat{J},\hat{\theta})$ be a {\bf vertical}, simply connected pseudohermitian submanifold of $H_{n}$ with CR dimension $m=n-1$. Then we have that\\
(i) if the second fundamental form $II\neq 0$ at each point, then the induced pseudohermitian structure $(\hat{J},\hat{\theta})$ constitute a complete set of invariants.\\
(ii) if $II=0$, then $M$ is an open part of $H_{n-1}=\{z_{n}=0\}$, after a Heisenberg rigid motion.\\

{\bf Theorem D.} Let $(M,\hat{J},\hat{\theta})$ be a {\bf completely non-vertical}, simply connected pseudohermitian submanifold of $H_{n}$ with CR dimension $m=n-1$. Then we have that\\
(i) the induced pseudohermitian structure $(\hat{J},\hat{\theta})$ constitute a complete set of invariants.\\
(ii) if the second fundamental form $II=0$ (or, equivalently, the pseudohermitian torsion $A_{jk}=0,\ 1\leq j,k\leq m$), then $M$ is an open part of the standard sphere $S^{2m+1}\subset H_{n}$, after a Heisenberg rigid motion.\\

Finally, in subsection \ref{pssuve}, we study the general properties of vertical pseudohermitian submanifolds and obtain\\

{\bf Theorem E.} Let $(M,\hat{J},\hat{\theta})$ be a {\bf vertical} pseudohermitian submanifold of $H_{n}$. Then we have that the Webster-torsion vanishes and the Webster-Ricci tensor is non-positive, as well as the pseudohermitian connection and the tangential connection coincide.\\

For the fundamental theorems, we used Cartan's method of moving frame as well as calculus on Lie groups. And we prove (ii) of Theorem D by means of the motions equation of the Darboux frame. Therefore, in section \ref{camemf}, we give a brief review of the Cartan's method of moving frame, which includes the motion equations and the structure equations.\\

{\bf Acknowledgments.} The author's research was supported in part by NCTS and in part by MOST 106-2115-M-007-017-MY3. He would like to thank Prof. Jih-Hsin Cheng, Prof. Jenn-Fang Hwang, and Prof. Paul Yang for regular kind encouragement and advising in his research.

\section{Cartan's method of moving frame}\label{camemf}
In this section, we give a brief review of Cartan's method of moving frame and Calculus on Lie groups. For the details, we refer reader to \cite{CL}. Let $(X,G)$ be a Klein geometry. The philosophy of Elie Cartan is that in many cases, the symmetry group $G$ may be identified with a set of frame on $X$. Then to investigate the geometry of a submanifold $M$ of $X$, one associate the submanifold with a natural set of frames. In this situation, the infinitesimal motion of this natural frame should  contain all the geometric  information of the submanifold $M$. Now we will go along the idea of Elie Cartan to get a complete set of invariants for $M$.

\subsection{The frames on $H_{n}$} 
An {\bf frame} for $H_n$ is a set of vectors of the form
\begin{equation*}
(p;e_{\beta},e_{n+\beta},T),
\end{equation*}
where $p\in H_n,\ e_{\beta}\in\xi(p)$ and $e_{n+\beta}=Je_{\beta}$, for $1\leq\beta\leq n$. In addition $\{e_{\beta},e_{n+\beta},T\}$ is an orthonormal frame with respect to the adapted metric $g_{\theta}$, which is defined by viewing the basis $\mathring{e}_{\beta},\mathring{e}_{n+\beta},T$ as an orthonormal basis.

\subsection{Identifying $PSH(n)$ with a set of frames}
We identify a symmetry $\Phi$ with a frame $(p;e_{\beta},e_{n+\beta},T)$, provided that $\Phi$ is the unique transformation on $H_n$ mapping the frame $(0;\mathring{e}_{\beta},\mathring{e}_{n+\beta},T)$ to the given frame $(p;e_{\beta},e_{n+\beta},T)$. That is,
\begin{equation*}
\begin{split}
\Phi_{*}(0;\mathring{e}_{\beta},\mathring{e}_{n+\beta},T)&=(\Phi(0);\Phi_{*}\mathring{e}_{\beta},\Phi_{*}\mathring{e}_{n+\beta},\Phi_{*}T)\\
&=(p;e_{\beta},e_{n+\beta},T)
\end{split}
\end{equation*}

\subsection{The matrix group representation of $PSH(n)$} If we identify points of $H_{n}$ and $1\times H_{n}$ by
\begin{equation*}
p\leftrightarrow
\left(\begin{array}{c}1\\p
\end{array}\right),
\end{equation*}
hence a vector $X\in TH_{n}$ can be identified by
\begin{equation*}
X\leftrightarrow
\left(\begin{array}{c}0\\X
\end{array}\right).
\end{equation*}
We thus identify $\Phi$ with a matrix $A\in GL(2n+2,R)$ by
\begin{equation}\label{grre}
\Phi\leftrightarrow (p;e_{\beta},e_{n+\beta},T)\leftrightarrow A=
\left(\begin{array}{cccc}1&0&0&0\\
p&e_{\beta}&e_{n+\beta}&T
\end{array}\right).
\end{equation}
We have 
\begin{equation*}
A\left(\begin{array}{c}1\\q
\end{array}\right)=\left(\begin{array}{c}1\\ \tilde{q}
\end{array}\right),\ \ \tilde{q}=\Phi(q).
\end{equation*}
This shows that (\ref{grre}) gives a matrix group representation of $PSH(n)$.

\subsection{The motion equations}
Let $\omega$ be the (left) Maurer Cartan form of $PSH(n)$. This is a $psh(n)$-valued one form defined by
\begin{equation}
\omega(v)=L_{g^{-1}*}v,
\end{equation}
for each $v\in T_{g}G$, where $G=PSH(n), g\in G$. That is, the Maurer Cartan form moves each vector $v$ to the identity element by the left translations. It is a natural way for us to identify each vector $v$ with a vector tangent to the identity. Since $PSH(n)$ has a matrix group representation, The Maurer cartan form has the simple elegant expression
\begin{equation}\label{maca}
\omega=A^{-1}dA,
\end{equation}
where $A\in PSH(n)$ is the moving point.
This formula (\ref{maca}) is equivalent to
\begin{equation}\label{moeq}
dA=A\omega,
\end{equation}
which is called the motion equations of the Heisenberg group. Taking the exterior derivative of the motion equation equation, we get the structure equations
\begin{equation}\label{steq}
d\omega+\omega\wedge\omega=0.
\end{equation}
\subsection{The Darboux frames for pseudohermitian submanifolds}
Let $U\subset M$ be an open subset. For each point $p\in U$, we always choose the frame $\{Z_{\beta},T\}$ such that $Z_{j}\in\hat{\xi}_{C}$ and $Z_{a}\in\hat{\xi}^{\perp}_{C}$, here $\hat{\xi}_{C}=\hat{\xi}\otimes C$ and $\hat{\xi}^{\perp}_{C}=\hat{\xi}^{\perp}\otimes C$. Such a kind of moving frame $p\rightarrow (p;Z_{\beta},T)$ is called the {\bf Darboux frame} (of complex version) over $U$.

Let $\{\theta^{\beta},\theta\}$ be the dual of $\{Z_{\beta},T\}$. Writing $Z_{\beta}=\frac{1}{2}(e_{\beta}-ie_{n+\beta})$ and $\theta^{\beta}=\omega^{\beta}+i\omega^{n+\beta}$. Then $\{e_{A},T\}$ and $\{\omega^{A},\theta\}$ are dual to each other. This frame field $p\rightarrow (p;e_{A},T)$ is the {\bf real version of the Darboux frame}. It is easy to see that $e_{k},e_{n+k}\in\hat{\xi}$, $e_{a},e_{n+a}\in\xi^{\perp}$, and $e_{n+\beta}=Je_{\beta}$.

Denoting $\hat{Z}_{j}=Z_{j}$ and writing $\hat{Z}_{j}=\frac{1}{2}(\hat{e}_{j}-i\hat{e}_{m+j})$. Then we have $\hat{e}_{j}=e_{j},\ \hat{e}_{m+j}=e_{n+j}$. Let $\{\hat{\omega}^{j}, \hat{\omega}^{m+j},\hat{\theta}\}$ be the dual of $\{\hat{e}_{j},\hat{e}_{m+j},\hat{T}\}$. We also denote $\hat{\theta}^{j}=\hat{\omega}^{j}+i\hat{\omega}^{m+j}$.

\subsection{The Darboux Derivative}
Let $f:U\rightarrow PSH(n)$ be a Darboux frame $f(p)=(p;e_{A},T)$. The Darboux derivative $\omega_{f}$ of $f$ is defined by 
\begin{equation}
\omega_{f}=\omega\circ f_{*}=f^{*}\omega.
\end{equation}
Therefore, it is just the usual differential $f_{*}$, provided that we have identified each vector with a vector to the identity element by left translations. From (\ref{maca}),
\begin{equation}
\begin{split}
\omega_{f}&=f^{*}\omega=f^{*}(A^{-1}dA)\\
&=(A\circ f)^{-1}d(A\circ f)=f^{-1}df,
\end{split}
\end{equation}
or, equivalently
\begin{equation}\label{moeqdrfr}
df=f\omega_{f}.
\end{equation}
This is the motion equations for the Darboux frame $f$. Again, taking the exterior derivative, we obtain the structure equations (the integrability conditions)
\begin{equation}
d\omega_{f}+\omega_{f}\wedge\omega_{f}=0.
\end{equation}
Writing
\begin{equation*}
f(p)=(p; e_{\beta}(p),e_{n+\beta}(p),T)
\end{equation*}
and
\begin{equation*}
\omega_{f}=\left(\begin{array}{cccc}
0 & 0 & 0 & 0 \\
\omega^{\beta} & \omega_{\alpha}{}^{\beta} & \omega_{n+\alpha}{}^{\beta} & 0 \\
\omega^{n+\beta} & \omega_{\alpha}{}^{n+\beta} & \omega_{n+\alpha}{}^{n+\beta} & 0 \\
\omega^{2n+1} & \omega^{n+\alpha} & -\omega^{\alpha} & 0
\end{array}\right).
\end{equation*}
Since $\omega_{f}$ is a $psh(n)$-valued one form, the entry forms satisfy
\begin{equation*}
\begin{split}
\omega_{a}{}^{b}&=-\omega_{b}{}^{a},\ \ \textrm{for}\ 1\leq a,b\leq 2n,\\
\omega_{n+\alpha}{}^{n+\beta}&=\omega_{\alpha}{}^{\beta},\ \omega_{\alpha}{}^{n+\beta}=-\omega_{n+\alpha}{}^{\beta},\ \ \  \textrm{for}\ 1\leq\alpha,\beta\leq n.
\end{split}
\end{equation*}

Then the motion equations and structures equations, respectively, become to be
\begin{equation}
\begin{split}
dp&=e_{\beta}\otimes\omega^{\beta}+e_{n+\beta}\otimes\omega^{n+\beta}+T\otimes\omega^{2n+1}\\
de_{\gamma}&=e_{\beta}\otimes\omega_{\gamma}{}^{\beta}+e_{n+\beta}\otimes\omega_{\gamma}{}^{n+\beta}+T\otimes\omega^{n+\gamma}\\
de_{n+\gamma}&=e_{\beta}\otimes\omega_{n+\gamma}{}^{\beta}+e_{n+\beta}\otimes\omega_{n+\gamma}{}^{n+\beta}-T\otimes\omega^{\gamma}\\
dT&=0;
\end{split}
\end{equation}
and 
\begin{equation}
\begin{split}
d\omega^{\beta}&=-\omega_{\alpha}{}^{\beta}\wedge\omega^{\alpha}-\omega_{n+\alpha}{}^{\beta}\wedge\omega^{n+\alpha}\\
d\omega^{n+\beta}&=-\omega_{\alpha}{}^{n+\beta}\wedge\omega^{\alpha}-\omega_{n+\alpha}{}^{n+\beta}\wedge\omega^{n+\alpha}\\
d\omega^{2n+1}&=2\sum_{\alpha=1}^{n}\omega^{\alpha}\wedge\omega^{n+\alpha}\\
d\omega_{\alpha}{}^{\beta}&=-\omega_{\gamma}{}^{\beta}\wedge\omega_{\alpha}{}^{\gamma}-\omega_{n+\gamma}{}^{\beta}\wedge\omega_{\alpha}{}^{n+\gamma}\\
d\omega_{n+\alpha}{}^{\beta}&=-\omega_{\gamma}{}^{\beta}\wedge\omega_{n+\alpha}{}^{\gamma}-\omega_{n+\gamma}{}^{\beta}\wedge\omega_{n+\alpha}{}^{n+\gamma},
\end{split}
\end{equation}

\subsection{The complex version}
Writing
\begin{equation*}
F(p)=(p; Z_{\beta}(p),T),\ \ \textrm{where}\ Z_{\beta}=\frac{1}{2}(e_{\beta}-ie_{n+\beta}),
\end{equation*}
and
\begin{equation}\label{cvdade}
\omega_{F}=\left(\begin{array}{ccc}
0 & 0 & 0  \\
\vartheta^{t}& \theta_{\gamma}{}^{\beta} & 0 \\
\theta& i\bar{\vartheta}& 0
\end{array}\right),
\end{equation}
where $\vartheta=(\theta^{1},\cdots,\theta^{n}),\ \theta^{\beta}=\omega^{\beta}+i\omega^{n+\beta},\ \theta_{\gamma}{}^{\beta}=\omega_{\gamma}{}^{\beta}+i\omega_{\gamma}{}^{n+\beta}$. And hence we have $\theta_{\gamma}{}^{\beta}+\theta_{\bar{\beta}}{}^{\bar{\gamma}}=0$. We have the complex version of motion equations 
\begin{equation}
\begin{split}
dp&=Z_{\beta}\otimes\theta^{\beta}+Z_{\bar\beta}\otimes\theta^{\bar\beta}+T\otimes\theta\\
dZ_{\gamma}&=Z_{\beta}\otimes\theta_{\gamma}{}^{\beta}+\frac{1}{2}T\otimes i\theta^{\bar\gamma}\\
dT&=0.
\end{split}
\end{equation}
And the structure equations is equivalent to
\begin{equation}
d\omega_{F}+\omega_{F}\wedge\omega_{F}=0,
\end{equation}
or
\begin{equation}
\begin{split}
d\theta^{\beta}&=\theta^{\gamma}\wedge\theta_{\gamma}{}^{\beta}\\
d\theta&=i\theta^{\gamma}\wedge\theta^{\bar{\gamma}}\\
d\theta_{\sigma}{}^{\beta}&=\theta_{\sigma}{}^{\gamma}\wedge\theta_{\gamma}{}^{\beta}.
\end{split}
\end{equation}

\subsection{Calculus on Lie groups}
Let $M$ be a {\bf simply connected} smooth manifold, $f:M\rightarrow PSH(n)$ be a smooth map. Recall that The (left) {\bf Darboux derivative} $\omega_{f}$ of $f$ is the $psh(n)$-valued $1$-form defined by $\omega_{f}=\omega\circ f_{*}=f^{*}\omega$. The Darboux derivative plays an important role in the theory of calculus on Lie groups. The fundamental theorems are Theorem \ref{tutthm} and Theorem \ref{tetthm}.

\begin{thm}[{\bf The uniqueness theorem}]\label{tutthm}
Let $f_{1},f_{2}:M\rightarrow PSH(n)$ be smooth maps. Then $\omega_{f_{1}}=\omega_{f_{2}}$ if and only if there exists $g\in PSH(n)$ such that $f_{2}(x)=g\cdot f_{1}(x)$ for all $x\in M$.\
\end{thm}
Theorem \ref{tutthm} says that two maps from $M$ into $PSH(n)$ are congruent with each other if and only if they have the same infinitesimal motions. Recall that $\omega_{f}$ satisfies the integrability conditions\begin{equation*}
d\omega_{f}+\omega_{f}\wedge\omega_{f}=0.
\end{equation*}
Conversely, one has
\begin{thm}[{\bf The existence theorem}]\label{tetthm}
Let $\eta$ be a $psh(n)$-valued one form on $M$ satisfying $d\eta+\eta\wedge\eta=0$. Then there is a smooth map $f:U\rightarrow PSH(n)$ such that $\eta|_{U}=\omega_{f}$.
\end{thm}
Theorem \ref{tetthm} totally depends on Frobenius Theorem. We will apply theorem \ref{tutthm} to the Darboux frames of pseudohermitian submanifolds. Then, to give Theorem A, we reduce to compute the Darboux derivatives of the Darboux frames. And using Theorem \ref{tetthm}, we obtain Theorem B. For the details about calculus on Lie groups, we refer reader to \cite{CCL},\cite{G},\cite{IL},\cite{PT} and \cite{S}.

\section{Local invariants of Pseudohermitian submanifolds}\label{loin} In this section, we define some geometric invariants for pseudohermitian submanifolds.
\subsection{The fundamental vector field $\nu$} 
\begin{prop}\label{fuve}
Let $(M,\hat{J},\hat{\theta})$ be a pseudohermitian submanifold of $H_{n}$. Then there exists a unique horizontal vector field $\nu\in\hat{\xi}^{\perp}$ such that $T+\nu\in TM$. Actually, denoting $\hat{T}=T+\nu$, it is not hard to see that  $\hat{T}$ is the Reeb vector field associated to $\hat{\theta}$.
\end{prop}
\begin{proof}
Let $\hat{T}=aT+\sum_{A=1}a^{A}e_{A}$, for some coefficients $a,a^{A}$. Since $1=\hat{\theta}(\hat{T})=\theta(\hat{T})=a$ and $\hat{T}\perp\hat{\xi}$, we have $\hat{T}=T+a^{a}e_{a}+a^{n+a}e_{n+a}$, and hence we can choose $\nu=a^{a}e_{a}+a^{n+a}e_{n+a}$. Next, suppose $\tilde{\nu}\in \hat{\xi}^{\perp}$ is another vector such that $T+\tilde{\nu}\in TM$. Then we have 
$\nu-\tilde{\nu}\in TM\cap\xi$, hence $\nu=\tilde{\nu}$.
\end{proof}

\begin{itemize}
\item If $\nu\equiv 0$, then $\hat{T}=T$, and hence we call $M^{2m+1}$ a {\bf vertical} submanifold.
\item If $\nu\neq 0$ at each point of $M$, then $M$ is {\bf completely non-vertical}.
\end{itemize}

\begin{prop}\label{reofcof}
We have 
\begin{equation}
\omega^{j}|_{M}=\hat{\omega}^{j},\ \ \omega^{n+j}|_{M}=\hat{\omega}^{m+j},\ \ \omega^{a}|_{M}=\frac{1}{2}\big<\nu,e_{a}\big>\hat{\theta},\ \ \omega^{n+a}|_{M}=\frac{1}{2}\big<\nu,e_{n+a}\big>\hat{\theta},
\end{equation}
where $\big<\ ,\ \big>$ is the Levi-metric, hence
\begin{equation}
\theta^{j}|_{M}=\hat{\theta}^{j}\ \ \ \theta^{a}|_{M}=\big<\nu,Z_{a}\big>\hat{\theta}.
\end{equation}
In particular, if $\nu=0$, then we have $\theta^{a}|_{M}=0$.
\end{prop}
\begin{proof}
We compute
\begin{equation}
\begin{split}
\omega^{j}(\hat{T})&=\omega^{j}(T+\nu)=0\\
\omega^{j}(\hat{e}_{k})&=\omega^{j}(e_{k})=\delta_{jk}\\
\omega^{j}(\hat{e}_{m+k})&=\omega^{j}(e_{n+k})=0,
\end{split}
\end{equation}
and
\begin{equation}
\begin{split}
\omega^{n+j}(\hat{T})&=\omega^{n+j}(T+\nu)=0\\
\omega^{n+j}(\hat{e}_{k})&=\omega^{n+j}(e_{k})=0\\
\omega^{n+j}(\hat{e}_{m+k})&=\omega^{n+j}(e_{n+k})=\delta_{jk}.
\end{split}
\end{equation}
Therefore $\{\omega^{j}|_{M},\omega^{n+j}|_{M},\theta|_{M}\}$ is the dual frame of $\{\hat{e}_{j},\hat{e}_{m+j}, \hat{T}\}$.
Similar computation shows that 
\begin{equation}
\omega^{a}|_{M}=\frac{1}{2}\big<\nu,e_{a}\big>\hat{\theta},\ \ \omega^{n+a}|_{M}=\frac{1}{2}\big<\nu,e_{n+a}\big>\hat{\theta}.
\end{equation}
\end{proof}

\subsection{The normal connection} The normal connection $\nabla^{\perp}$ which is defined, on the normal complex bundle $\hat{\xi}^{\perp}\otimes C$ spanned by $Z_{a}$, by
\begin{equation}
\nabla^{\perp}Z_{a}=\theta_{a}{}^{b}\otimes Z_{b},
\end{equation}
which is the orthogonal projection of the pseudohermitian connection $\nabla Z_{a}$ onto the normal bundle.

\subsection{The tangential connection} The tangential connection $\nabla^{t}$ which is defined, on the complex bundle $\hat{\xi}_{C}$ spanned by $Z_{j}$, by
\begin{equation}
\nabla^{\perp}Z_{j}=\theta_{j}{}^{k}\otimes Z_{k},
\end{equation}
which is the orthogonal projection of the pseudohermitian connection $\nabla Z_{j}$ onto the contact bundle.

\begin{itemize}
\item Let $\hat{\theta}_{j}{}^{k}$ be the pseudohermitian connection forms with respect to the frame field $Z_{j}$. Then, from (\ref{incon2}), we have
\begin{equation}
\theta_{j}{}^{k}|_{M}=\hat{\theta}_{j}{}^{k}+i\delta_{jk}|\nu|^{2}\hat{\theta}.
\end{equation}
Therefore, in general, $\nabla^{t}\neq\nabla^{p.h.}$, the associated pseudohermitian connection of $M$.
\item If $M$ is vertical, then $\nabla^{t}=\nabla^{p.h.}$. 
\end{itemize}

\subsection{The second fundamental form} Define the bilinear form $II^a$ on $\hat{\xi}_{1,0}$ by
\begin{equation}
II^{a}(X,Y)=-\big<X,\nabla_{\overline{Y}}Z_{a}\big>.
\end{equation}

\begin{itemize}
\item We have $II^{a}=\theta^{j}\otimes\theta_{j}{}^{a}$. If $\widetilde{Z}_{a}=C_{a}{}^{b}Z_{b}$ is another normal frame field, then $\widetilde{II}^{a}=C_{\bar a}{}^{\bar b}II^{b}$.
\item $II^{a}\otimes Z_{a}$ is independent of the choice of the normal frame field $Z_{a}$.
\end{itemize}

{\bf The second fundamental form} $II$ for $M$ is defined, to be a map 
\begin{equation}
II:\hat{\xi}_{1,0}\times\hat{\xi}_{1,0}\rightarrow\hat{\xi}^{\perp}_{1,0}, 
\end{equation}
 by
\begin{equation}
II=II^{a}\otimes Z_{a}=\theta^{j}\otimes\theta_{j}{}^{a}\otimes Z_{a}.
\end{equation}

\section{General properties}\label{tgp}

\subsection{Pseudohermitian submanifolds with $\nu\equiv 0$}\label{pssuve} The canonical example is the Heisenberg subgroup $H_{m}$ which is defined by $H_{m}=\{(z,t)\in H_{n}\ |\ z_{a}=0\}$. Now we discuss the general properties of such kind of submanifolds. From Proposition \ref{reofcof}, we have 
\begin{equation}
\theta^{j}|_{M}=\hat{\theta}^{j},\ \ \textrm{and}\ \ \theta^{a}|_{M}=0.
\end{equation}
Therefore, we have the structure equations
\begin{equation}\label{sefor0}
\begin{split}
d\theta^{j}&=\theta^{k}\wedge\theta_{k}{}^{j}\\
0&=\theta^{k}\wedge\theta_{k}{}^{a}\ \ \ \ (\because \theta^{a}=0)\\
d\theta&=i\theta^{k}\wedge\theta^{\bar k}\\
d\theta_{j}{}^{l}&=\theta_{j}{}^{k}\wedge\theta_{k}{}^{l}+\theta_{j}{}^{c}\wedge\theta_{c}{}^{l}\\
d\theta_{j}{}^{a}&=\theta_{j}{}^{k}\wedge\theta_{k}{}^{a}+\theta_{j}{}^{c}\wedge\theta_{c}{}^{a}\\
d\theta_{a}{}^{j}&=\theta_{a}{}^{k}\wedge\theta_{k}{}^{j}+\theta_{a}{}^{c}\wedge\theta_{c}{}^{j}\\
d\theta_{a}{}^{b}&=\theta_{a}{}^{k}\wedge\theta_{k}{}^{b}+\theta_{a}{}^{c}\wedge\theta_{c}{}^{b},
\end{split}
\end{equation}

\begin{itemize}
\item From the first equation $d\theta^{j}=\theta^{k}\wedge\theta_{k}{}^{j}$ of (\ref{sefor0}), together with $\theta_{k}{}^{j}+\theta_{\bar j}{}^{\bar k}$, we have 
\begin{equation}\hat{\tau}^{j}\equiv 0,\ \hat{\theta}_{k}{}^{j}=\theta_{k}{}^{j},
\end{equation}
where $\hat{\tau}^{j},\ \hat{\theta}_{k}{}^{j}$ are the pseudohermitian torsion forms and connection forms with respect to the admissible coframe $\{\hat{\theta}^{j}\}$.
\item From the second equation $0=\theta^{k}\wedge\theta_{k}{}^{a}$ of (\ref{sefor0}), together with Cartan lemma, we have
\begin{equation}
\theta_{j}{}^{a}=h^{a}_{jk}\theta^{k},
\end{equation}
for some functions $h^{a}_{jk}$ satisfying $h^{a}_{jk}=h^{a}_{kj}$. Therefore
\begin{equation}\label{secfun2}
\begin{split}
II&=\theta^{j}\otimes\theta_{j}{}^{a}\otimes Z_{a}\\
&=h^{a}_{jk}\theta^{j}\otimes\theta^{k}\otimes Z_{a},
\end{split}
\end{equation}
\item The fourth equation of (\ref{sefor0})
\begin{equation}
d\theta_{j}{}^{l}=\theta_{j}{}^{k}\wedge\theta_{k}{}^{l}+\theta_{j}{}^{c}\wedge\theta_{c}{}^{l}
\end{equation}
is called the {\bf Gauss-like equation}. Since $\theta_{j}{}^{k}=\hat{\theta}_{j}{}^{k}$ and $\theta_{j}{}^{c}=h^{c}_{jk}\theta^{k}$, it is easy to see that the Gauss-like equation is equivalent to 
\begin{equation}
R_{j\bar{l}\zeta\bar{\eta}}=-\sum_{c}h^{c}_{j\zeta}h^{\bar c}_{\bar{l}\bar{\eta}},
\end{equation}
which implies $R_{\zeta\bar{\eta}}=-\sum_{c=m+1}^{n}h^{c}_{k\zeta}h^{\bar c}_{\bar{k}\bar{\eta}}$, and hence the Webster-Ricci tensor is {\bf non-positive}.
\item The fifth equation of (\ref{sefor0}) 
\begin{equation}
d\theta_{j}{}^{a}=\theta_{j}{}^{k}\wedge\theta_{k}{}^{a}+\theta_{j}{}^{c}\wedge\theta_{c}{}^{a}
\end{equation}
is equivalent to the sixth equation of (\ref{sefor0}) 
\begin{equation}
d\theta_{a}{}^{j}=\theta_{a}{}^{k}\wedge\theta_{k}{}^{j}+\theta_{a}{}^{c}\wedge\theta_{c}{}^{j}.
\end{equation}
Either one is called the {\bf Codazzi-like equation}.
\item The last equation of (\ref{sefor0}) 
\begin{equation}
d\theta_{a}{}^{b}=\theta_{a}{}^{k}\wedge\theta_{k}{}^{b}+\theta_{a}{}^{c}\wedge\theta_{c}{}^{b}
\end{equation}
is called the {\bf Ricci-like equation}.
\end{itemize}

\subsection{Pseudohermitian submanifolds with $\nu$ nowhere zero}\label{stspsp} The canonical example is the standard sphere $S^{2n-1}\subset C^{n}\subset H_{n}=C^{n}\times R,\ n\geq 2$. It is defined by 
\begin{equation*}
S^{2n-1}(r)=\left\{(z_{1},\cdots,z_{n},0)\in H_{n}\ |\ \sum_{\beta=1}^{n}z_{\beta}z_{\bar\beta}=r^{2}\right\}.
\end{equation*}
Let $L_{p}$ be a left translation, we compute the image of $(z,0)\in S^{2n-1}(r)$,
\begin{equation*}
L_{p}(z,0)=p+x_{\beta}\mathring{e}_{\beta}(p)+y_{\beta}\mathring{e}_{n+\beta}(p),
\end{equation*}
where $z_{\beta}=x_{\beta}+iy_{\beta}$, and hence the image of $S^{2n-1}(r)$ under $L_{p}$ is
\begin{equation}
L_{p}\left(S^{2n-1}(r)\right)=\{q\in H_{n}\  |\ q-p\in\xi(p),\ \textrm{and}\ |q-p|=r\},
\end{equation}
where the norm $|\cdot|$ is measured by the levi metric. Next, there are two pseudohermitian structures induced on $S^{2n-1}$, one is from the Heisenberg group $H_{n}$, denoted by $(\hat{J},\hat{\theta})$, and the other is from $C^{n}$. It is easy to see that these two induced pseudohermitian structures coincide on $S^{2n-1}(r)$ as the following specifies. Let $u=\left(\sum_{\beta=1}^{n}z_{\beta}z_{\bar\beta}\right)-r^{2}$ be the defining function. We have
\begin{equation}\label{coform}
\hat{\theta}=\Theta|_{S^{2n-1}}=x_{\beta}dy_{\beta}-y_{\beta}dx_{\beta}=\frac{i(\bar{\partial}u-\partial u)}{2},
\end{equation}
hence
\begin{equation}
\hat{\xi}=\ker{\hat{\theta}}=TS^{2n-1}\cap J_{C^{n}}(TS^{2n-1})\subset TC^{n},
\end{equation}
where $J_{C^{n}}$ is the standard complex structure on $C^{n}$. 
\begin{lem}
Let $p=(z,t)\in S^{2n-1}$. If a vector $X=a_{\beta}\mathring{e}_{\beta}+a_{n+\beta}\mathring{e}_{n+\beta}\in\hat{\xi}(p)$, then $a_{\beta}y_{\beta}-a_{n+\beta}x_{\beta}=0$, where $z_{\beta}=x_{\beta}+iy_{\beta}$. In addition, we have
\begin{equation}\label{exofve}
X=a_{\beta}\mathring{e}_{\beta}+a_{n+\beta}\mathring{e}_{n+\beta}=a_{\beta}\frac{\partial}{\partial x_{\beta}}+a_{n+\beta}\frac{\partial}{\partial y_{\beta}},
\end{equation}
for all $X\in\hat{\xi}$.
\end{lem}
\begin{proof} We compute
\begin{equation}
\begin{split}
X&=a_{\beta}\mathring{e}_{\beta}+a_{n+\beta}\mathring{e}_{n+\beta}\\
&=a_{\beta}\frac{\partial}{\partial x_{\beta}}+a_{n+\beta}\frac{\partial}{\partial y_{\beta}}+(a_{\beta}y_{\beta}-a_{n+\beta}x_{\beta})\frac{\partial}{\partial t}.
\end{split}
\end{equation}
Since $\hat{\xi}\subset TC^{n}$, we get $a_{\beta}y_{\beta}-a_{n+\beta}x_{\beta}=0$.
\end{proof}
For all $X\in\hat{\xi}$, 
\begin{equation}
\begin{split}
\hat{J}(X)&=J(a_{\beta}\mathring{e}_{\beta}+a_{n+\beta}\mathring{e}_{n+\beta})=a_{\beta}\mathring{e}_{n+\beta}-a_{n+\beta}\mathring{e}_{\beta}\\
&=a_{\beta}\frac{\partial}{\partial y_{\beta}}-a_{n+\beta}\frac{\partial}{\partial x_{\beta}}=J_{C^{n}}\left(a_{\beta}\frac{\partial}{\partial x_{\beta}}+a_{n+\beta}\frac{\partial}{\partial y_{\beta}}\right)\\
&=J_{C^{n}}(X),
\end{split}
\end{equation}
which shows that $\hat{J}$ is also induced from $J_{C^{n}}$.
On the other hand, from (\ref{coform}), we have 
\begin{equation}
\hat{T}=\frac{i\left(z_{\beta}\frac{\partial}{\partial z_{\beta}}-z_{\bar\beta}\frac{\partial}{\partial z_{\bar \beta}}\right)}{r^{2}}=\frac{\partial}{\partial t}+\nu,
\end{equation}
which implies that
\begin{equation}
\begin{split}
\nu&=\frac{i\left(z_{\beta}\frac{\partial}{\partial z_{\beta}}-z_{\bar\beta}\frac{\partial}{\partial z_{\bar \beta}}\right)}{r^{2}}-\frac{\partial}{\partial t}=\frac{i\left(x_{\beta}\frac{\partial}{\partial y_{\beta}}-y_{\bar\beta}\frac{\partial}{\partial x_{\bar \beta}}\right)}{r^{2}}-\frac{\partial}{\partial t}\\
&=\frac{x_{\beta}\mathring{e}_{n+\beta}-y_{\beta}\mathring{e}_{\beta}}{r^{2}}.
\end{split}
\end{equation}
This shows that the standard sphere $S^{2n-1}(r)$ is completely non-vertical.

\section{The uniqueness theorem}\label{tut}
In this section, we are going to prove Theorem A. Let $M$ and $N$ be two pseudohermitian submanifolds with the same CR dimension $m$. Suppose $\Phi$ is a Heisenberg rigid motion such that $\Phi(M)=N$ and denote $\varphi=\Phi|_{M}$.

Let $\{Z_{\beta}\}$ be a frame field over $M$, and suppose $\widetilde{Z}_{\beta}=\Phi_{*}Z_{\beta}$, the set $\{\widetilde{Z}_{\beta}\}$ is a frame field over $N$. Suppose $\{\theta^{\beta},\Theta\}$ and $\{\widetilde{\theta}^{\beta},\Theta\}$ are the dual frame fields of $\{Z_{\beta},T\}$ and $\{\widetilde{Z}_{\beta},T\}$, respectively. Then we have
\begin{equation}
\theta^{\beta}=\Phi^{*}\widetilde{\theta}^{\beta},\ \ \Theta=\Phi^{*}\Theta.
\end{equation}
In particular, we have
\begin{equation}\label{prepseu}
\theta^{j}=\varphi^{*}\widetilde{\theta}^{j},\ \ \hat{\theta}=\varphi^{*}\widetilde{\theta},
\end{equation}
where $\hat{\theta}$ and $\widetilde{\theta}$ are the induced contact form on $M$ and $N$, respectively. (\ref{prepseu}) implies that $\varphi$ preserves the induced pseudohermitian structures. 
 
 From the structure equation on $H_{n}$, we compute
 \begin{equation}
 \begin{split}
 d\theta^{\beta}&=\theta^{\gamma}\wedge\theta_{\gamma}{}^{\beta}=\left(\Phi^{*}\widetilde{\theta}^{\gamma}\right)\otimes\theta_{\gamma}{}^{\beta}\\
\Vert\ \ & \\
d(\Phi^{*}\widetilde{\theta}^{\beta})&= \Phi^{*}(d\widetilde{\theta}^{\beta}),
 \end{split}
 \end{equation}
which is equivalent to 
\begin{equation}
d\widetilde{\theta}^{\beta}=\widetilde{\theta}^{\gamma}\otimes(\Phi^{-1})^{*}\theta_{\gamma}{}^{\beta}.
\end{equation}
Together with
\begin{equation}
(\Phi^{-1})^{*}\theta_{\gamma}{}^{\beta}+(\Phi^{-1})^{*}\theta_{\bar\beta}{}^{\bar\gamma}=(\Phi^{-1})^{*}(\theta_{\gamma}{}^{\beta}+\theta_{\bar\beta}{}^{\bar\gamma})=0,
\end{equation}
and by the uniqueness, we get 
\begin{equation}
\theta_{\gamma}{}^{\beta}=\Phi^{*}\widetilde{\theta}_{\gamma}{}^{\beta}.
\end{equation}
In particular, we have
\begin{equation}
\theta_{\gamma}{}^{\beta}=\varphi^{*}\widetilde{\theta}_{\gamma}{}^{\beta}, 
\end{equation}
and hence 
\begin{equation}
\big<II,V\big>=\varphi^{*}\big<\widetilde{II},\Phi_{*}V\big>,
\end{equation}
for all $V\in\hat{\xi}^{\perp}_{C}$.

The defferential $\Phi_{*}$ defines a vector bundle isomorphism
\begin{equation}
\begin{array}{ccc}
\hat{\xi}^{\perp}_{1,0}&\longrightarrow&\widetilde{\xi}^{\perp}_{1,0}\\
\downarrow& &\downarrow\\
M&\longrightarrow&N,
\end{array}
\end{equation}
which preserving the hermitian structures induced from the levi-metric and cover $\varphi$, such that $\Phi_{*}$ preserves the normal connections, i.e.,
\begin{equation}
\Phi_{*}(\nabla^{\perp}_{X}Z_{a})=\widetilde{\nabla}^{\perp}_{\varphi_{*}X}(\Phi_{*}Z_{a}),
\end{equation}
for all $X\in TM$, where $\nabla^{\perp}$ and $\widetilde{\nabla}^{\perp}$ are the induced normal connections on $M$ and $N$, respectively. Finally, it is easy to see that $\Phi_{*}\nu=\widetilde{\nu}$.

\begin{de}\label {defof1}
Suppose that $M$ and $N$ are two pseudohermitian submanifolds of $H_{n}$ with the same CR dimension $m$. We say that $M$ and $N$ have {\bf the same (induced) pseudohermitian structures, the second fundamental forms, the normal connections and the fundamental vector fields} if there exists a vector bundle isomorphism $F:\hat{\xi}^{\perp}_{1,0}\rightarrow\widetilde{\xi}^{\perp}_{1,0}$, which preserves the induced hermitian structures and covers a map $\varphi:M\rightarrow N$, such that 
\begin{itemize}
\item $F$ preserves the induced pseudohermitian structures: $\varphi_{*}\circ \hat{J}=\widetilde{J}\circ\varphi_{*}$;\ and\ $\varphi^{*}\widetilde{\theta}=\hat{\theta}$;
\item $F$ preserves the second fundamental forms: $\big<II,V\big>_{\hat{\xi}^{\perp}_{1,0}}=\varphi^{*}\big<\widetilde{II},FV\big>_{\widetilde{\xi}^{\perp}_{1,0}}$,\ for all $V\in\hat{\xi}^{\perp}_{1,0}$.
\item $F$ preserves the normal connections: $F(\nabla^{\perp}_{X}V)=\widetilde{\nabla}^{\perp}_{\varphi_{*}X}(FV)$,\ for all $X\in TM,\ V\in\hat{\xi}^{\perp}_{1,0}$.
\item $F$ preserves the fundamental vector field: $F\nu=\widetilde{\nu}$.
\end{itemize}
\end{de}
Therefore we conclude that if $M$ is congruent with $N$, then they have the same such four invariants. Conversely, we have 
\begin{thm}\label{mainthm01}
Let $(M,\hat{J},\hat{\theta})$ and $(N,\widetilde{J},\widetilde{\theta})$ be two simply connected pseudohermitian submanifolds of $H_{n}$ with CR dimension $m$. Suppose that they have the same (induced) pseudohermitian structures, the second fundamental forms, the normal connections and the fundamental vector fields. Then they differ by a Heisenberg rigid motion.
\end{thm}

\begin{cor}
If $M$ and $N$ are {\bf vertical}, then the (induced) pseudohermitian structures, the second fundamental forms and the normal connections constitute a complete set of invariants.
\end{cor}

\subsection{The proof of Theorem \ref{mainthm01}}
Let $(M,\hat{J},\hat{\theta})$ be a pseudohermitian submanifold of $H_n$. Recall that we always choose the frame field $\{Z_{\beta},T\}$ over $M$ such that $Z_{j}\in\hat{\xi}_{1,0}$ and $Z_{a}\in\hat{\xi}^{\perp}_{1,0}$. This is a Darboux frame. Let $\{\theta^{\beta},\theta\}$ be the dual of $\{Z_{\beta},T\}$.
We would like to show that the restrictions of $\theta^{\beta}$ and $\theta_{\beta}{}^{\gamma}$ to $M$ are expressed as the following:
\begin{equation}\label{incon2}
\begin{split}
\theta^{j}|_{M}&=\hat{\theta}^{j},\\
\theta^{a}|_{M}&=\big<\nu,Z_{a}\big>\hat{\theta},\\
\theta|_{M}&=\hat{\theta},\\
\theta_{j}{}^{k}|_{M}&=\hat{\theta}_{j}{}^{k}+i\delta_{jk}|\nu|^{2}\hat{\theta},\\
\theta_{j}{}^{a}|_{M}&=h^{a}_{jk}\hat{\theta}^{k}+i\delta_{jk}\big<\nu,Z_a\big>\hat{\theta}^{\bar k}+\big<\nabla^{\perp}_{\hat{Z}_j}\nu,Z_a\big>\hat{\theta};
\end{split}
 \end{equation}
and here $h^{a}_{jk}=II^{a}(\hat{Z}_{j},\hat{Z}_{k})$, and $\theta_{a}{}^{b}|_{M}$ is the normal connection forms w.r.t. $\{Z_{a}\}$. This shows that the Darboux derivative of the Draboux frame is completely determined by the induced pseudohermitian structures, the second fundamental forms, the normal connections and the fundamental vector fields.

Now we prove (\ref{incon2}).
\begin{equation}\label{incon3}
\begin{split}
d\theta^{a}&=\theta^{j}\wedge\theta_{j}{}^{a}+\theta^{b}\wedge\theta_{b}{}^{a}\\
&=\hat{\theta}^{j}\wedge\theta_{j}{}^{a}+\big<\nu,Z_{b}\big>\hat{\theta}\wedge\theta_{b}{}^{a}.
\end{split}
\end{equation}
On the other hand,
\begin{equation}\label{incon4}
\begin{split}
d\theta^{a}&=d\left(\big<\nu,Z_{a}\big>\hat{\theta}\right)=d\big<\nu,Z_{a}\big>\wedge\hat{\theta}+\big<\nu,Z_{a}\big>d\hat{\theta}\\
&=\left(\big<\nabla\nu,Z_{a}\big>+\big<\nu,\nabla_{\bar{\cdot}}Z_{a}\big>\right)\wedge\hat{\theta}+\big<\nu,Z_{a}\big>d\hat{\theta}\\
&=\left(\big<\nabla^{\perp}\nu,Z_{a}\big>+\big<\nu,\nabla^{\perp}_{\bar{\cdot}}Z_{a}\big>\right)\wedge\hat{\theta}+\big<\nu,Z_{a}\big>d\hat{\theta},
\end{split}
\end{equation}
and 
\begin{equation}\label{incon5}
\big<\nu,\nabla^{\perp}_{\bar{\cdot}}Z_{a}\big>=\big<\nu,\theta_{a}{}^{b}(\bar{\cdot})\otimes Z_{b}\big>=\big<\nu,Z_{b}\big>\theta_{\bar a}{}^{\bar b}
\end{equation}
From (\ref{incon3}),(\ref{incon4}) and (\ref{incon5}), we obtain
\begin{equation}
\hat{\theta}^{j}\wedge\theta_{j}{}^{a}=\big<\nabla^{\perp}\nu,Z_{a}\big>\wedge\hat{\theta}+\big<\nu,Z_{a}\big>d\hat{\theta}.
\end{equation}
That is,
\begin{equation}
\begin{split}
&\theta_{j}{}^{a}(\hat{Z}_{k})\hat{\theta}^{j}\wedge\hat{\theta}^{k}+\theta_{j}{}^{a}(\hat{Z}_{\bar k})\hat{\theta}^{j}\wedge\hat{\theta}^{\bar k}+\theta_{j}{}^{a}(\hat{T})\hat{\theta}^{j}\wedge\hat{\theta}\\
=&\big<\nabla_{\hat{Z}_{k}}^{\perp}\nu,Z_{a}\big>\hat{\theta}^{k}\wedge\hat{\theta}+\big<\nabla_{\hat{Z}_{\bar k}}^{\perp}\nu,Z_{a}\big>\hat{\theta}^{\bar k}\wedge\hat{\theta}+i\big<\nu,Z_{a}\big>\hat{\theta}^{j}\wedge\hat{\theta}^{\bar j},
\end{split}
\end{equation}
which implies
\begin{equation}
\begin{split}
\theta_{j}{}^{a}(\hat{T})&=\big<\nabla_{\hat{Z}_{ j}}^{\perp}\nu,Z_{a}\big>\\
0&=\big<\nabla_{\hat{Z}_{\bar j}}^{\perp}\nu,Z_{a}\big>\\
\theta_{j}{}^{a}(\hat{Z}_{\bar k})&=i\delta_{jk}\big<\nu,Z_{a}\big>\\
\theta_{j}{}^{a}(\hat{Z}_{ k})&=\theta_{k}{}^{a}(\hat{Z}_{j})=h^{a}_{jk},
\end{split}
\end{equation}
and thus
\begin{equation}
\theta_{j}{}^{a}=h^{a}_{jk}\hat{\theta}^{k}+i\delta_{jk}\big<\nu,Z_a\big>\hat{\theta}^{\bar k}+\big<\nabla^{\perp}_{\hat{Z}_j}\nu,Z_a\big>\hat{\theta}.
\end{equation}
Now we compute
\begin{equation}\label{incon6}
\begin{split}
d\theta^{k}&=\theta^{j}\wedge\theta_{j}{}^{k}+\theta^{a}\wedge\theta_{a}{}^{k}\\
&=\hat{\theta}^{j}\wedge\theta_{j}{}^{k}+\hat{\theta}\wedge\left(\big<\nu,Z_{a}\big>\theta_{a}{}^{k}\right).
\end{split}
\end{equation}
On the other hand,
\begin{equation}\label{incon7}
d\theta^{k}=d\hat{\theta}^{k}=\hat{\theta}^{j}\wedge\hat{\theta}_{j}{}^{k}+\hat{\theta}\wedge\tau^{k}.
\end{equation}
By Cartan lemma, there exists functions $B_{jl}^{k},B_{j(m+1)}^{k},B_{(m+1)l}^{k}$ and $B_{(m+1)(m+1)}^{k}$ such that
\begin{equation}\label{incon8}
\begin{split}
\hat{\theta}_{j}{}^{k}&=\theta_{j}{}^{k}+B_{jl}^{k}\hat{\theta}^{l}+B_{j(m+1)}^{k}\hat{\theta}\\
\hat{\tau}^{k}&=\left(\big<\nu,Z_{a}\big>\theta_{a}{}^{k}\right)+B_{(m+1)l}^{k}\hat{\theta}^{l}+B_{(m+1)(m+1)}^{k}\hat{\theta},
\end{split}
\end{equation}
where $B_{jl}^{k}=B_{lj}^{k}$ and $B_{j(m+1)}^{k}=B_{(m+1)j}^{k}$, for $1\leq j,k,l\leq m$. 
Since $\hat{\tau}^{k}=A^{k}{}_{\hat{l}}\hat{\theta}^{\bar l}$, comparing with (\ref{incon8}), we get
\begin{equation}\label{incon9}
\begin{split}
A^{k}{}_{\bar l}&=\sum_{a=m+1}^{n}\big<\nu,Z_{a}\big>\theta_{a}{}^{k}(\hat{Z}_{\bar l})=-\sum_{a=m+1}^{n}\big<\nu,Z_{a}\big>h^{\bar a}_{\bar{k}\bar{l}},\\
B_{(m+1)l}^{k}&=-\sum_{a=m+1}^{n}\big<\nu,Z_{a}\big>\theta_{a}{}^{k}(\hat{Z}_{l})=-i\delta_{kl}|\nu|^{2},\\
B_{(m+1)(m+1)}^{k}&=-\sum_{a=m+1}^{n}\big<\nu,Z_{a}\big>\theta_{a}{}^{k}(\hat{T})=\sum_{a=m+1}^{n}\big<\nu,Z_{a}\big>\big<\nabla^{\perp}_{\hat{Z}_{\bar k}}\nu,Z_{\bar a}\big>.
\end{split}
\end{equation}
Finally, since $\theta_{j}{}^{k}+\theta_{\bar k}{}^{\bar j}=0$ and $\hat{\theta}_{j}{}^{k}+\hat{\theta}_{\bar k}{}^{\bar j}=0$, we have, from (\ref{incon8}), $B_{jl}^{k}=0$, and hence
\begin{equation}
\hat{\theta}_{j}{}^{k}=\theta_{j}{}^{k}-i\delta_{jk}|\nu|^{2}\hat{\theta}.
\end{equation}
This completes the proof of (\ref{incon2}).

\subsection{The integrability condition}\label{tic}
Let $(M,\hat{J},\hat{\theta})$ be a pseudohermitian submanifold of $H_n$. We choose a Darboux frame $\{Z_{\beta},T\}$ over $M$. Let $\{\theta^{\beta},\theta\}$ be the dual of $\{Z_{\beta},T\}$.
\begin{de} 
The {\bf restriction} to $M$ of the structure equations of $H_n$, 
\begin{equation}\label{incon01}
\begin{split}
d\theta^{\beta}&=\theta^{\gamma}\wedge\theta_{\gamma}{}^{\beta}\\
d\theta&=i\theta^{\gamma}\wedge\theta^{\bar{\gamma}}\\
d\theta_{\sigma}{}^{\beta}&=\theta_{\sigma}{}^{\gamma}\wedge\theta_{\gamma}{}^{\beta},
\end{split}
\end{equation}
is defined to be the {\bf integrability conditions} of $M$. Note that the restrictions of $\theta^{\beta}$ and $\theta_{\beta}{}^{\gamma}$ to $M$ have the expressions of the forms as (\ref{incon2}) specifies.. 
\end{de}

\section{The existence theorem}\label{tet}
In this section, we would like to show Theorem B.
Let $(M,J_{M},\theta_{M})$ be a pseudohermitian manifold with CR dimension $m$. Since the existence theory is local, we assume that $M$ is simply connected. Putting $\xi_{M}=\textrm{ker}\theta_M$ and $\eta=\theta_{M}$. 
\begin{itemize}
\item  Let $\xi^\perp_M$ be a complex vector bundle over $M$, of complex dimension $n-m$, with a Hermitian metric $h^\perp_M$ and a connection $\nabla^M$ compatible with $h^\perp_M$.
\item Suppose $\{W_{1},\cdots, W_{m}\in T_{1,0}M\}$ is an orthonormal CR holomorphic frame field of $M$. Its dual is denoted by $\{\eta^1,\cdots,\eta^m\}$. Let $\hat{\eta}_{j}{}^k$ be the pseudohermitian conection forms w.r.t. $W_j$. We have $\hat{\eta}_{j}{}^k+\hat{\eta}_{\bar k}{}^{\bar j}=0$.
\item Suppose $\{W_{m+1},\cdots, W_{n}\}$ is an orthonormal frame field of $\xi^\perp_M$ w.r.t. $h^\perp_M$ and $\eta_{a}{}^{b}$ are the connection forms w.r.t. $\{W_a\}$, i.e.,
\begin{equation}
\nabla^{M}W_a=\eta_a{}^b\otimes W_b.
\end{equation}
We have $\eta_a{}^b+\eta_{\bar b}{}^{\bar a}=0$.
\item Let $II_M: T_{1,0}M\times T_{1,0}M\rightarrow\xi^{\perp}_M$ be a $\xi^{\perp}_M$-valued complex bilinear form.
\item Let $\mu$ be a real section of the bundle $\xi^{\perp}_{M}$ over $M$. And define 
\begin{equation}
\begin{split}
\eta^{a}&=\big<\mu,W_{a}\big>\eta,\\
\eta_{j}{}^{k}&=\hat{\eta}_{j}{}^{k}+i\delta_{jk}|\mu|^{2}\eta,\\
\eta_j{}^a&=g^{a}_{jk}\eta^{k}+i\delta_{jk}\big<\mu,W_a\big>\eta^{\bar k}+\big<\nabla^{M}_{W_j}\mu,W_a\big>\eta,
\end{split}
\end{equation}
where $\big<\ ,\ \big>=h_{M}^{\perp}$ and $g^{a}_{jk}=\big<II_{M}(W_{j},W_{k}),W_{a}\big>$.
\item Finally, define $\eta_{a}{}^{j}$ by $\eta_{j}{}^{a}+\eta_{\bar a}{}^{\bar j}=0$.
\end{itemize}

\begin{thm}\label{mainthm02}
Suppose that the pseudohermitian manifold $(M^{2m+1},J_M,\theta_M)$, together with $II_M$, the compatible connection $\nabla^M$ and the section $\mu$ satisfies the {\bf integrability conditions}, in the sense that $\eta^{\beta}, \eta$ and $\eta_{\gamma}{}^{\beta}$ satisfy (\ref{incon01}).Then
\begin{itemize}
\item There exists an embedding $\phi$ such that $(M,J_M,\theta_M)$ can be embedded into $H_{n}$ with CR dimensionn m. 
\item In addition, there exists a vector bundle isomorphism $\Psi:\xi^{\perp}_M\rightarrow\hat{\xi}^\perp_{1,0}$, covering $\phi$, such that $\Psi^*II=II_M$,\ $\Psi^*\nabla^{\perp}=\nabla^M$, and $\Psi^{*}\nu=\mu$,
where $\hat{\xi}^\perp_{1,0},II,\nabla^\perp$ and $\nu$ are, respectively, the induced normal bundle, second fundamental form, normal connection and fundamental vector field $\nu$ over $\phi(M)$. 
\end{itemize}
\end{thm}
\begin{proof}
Let $\flat=(\eta^{1},\cdots,\eta^{n}), \bar{\flat}=(\eta^{\bar 1},\cdots,\eta^{\bar n})$. Define the matrix $\Pi$ by
\begin{equation}\label{cvdade1}
\Pi=\left(\begin{array}{ccc}
0 & 0 & 0  \\
\flat^{t}& \eta_{\gamma}{}^{\beta} & 0 \\
\eta& i\bar{\flat}& 0
\end{array}\right).
\end{equation}
The integrability conditions means that 
\begin{equation}\label{incon10}
d\Pi+\Pi\wedge\Pi=0.
\end{equation}
Taking the real version $\zeta$ of $\Pi$,
\begin{equation*}
\zeta=\left(\begin{array}{cccc}
0 & 0 & 0 & 0 \\
\lambda^{\beta} & \lambda_{\alpha}{}^{\beta} & \lambda_{n+\alpha}{}^{\beta} & 0 \\
\lambda^{n+\beta} & \lambda_{\alpha}{}^{n+\beta} & \lambda_{n+\alpha}{}^{n+\beta} & 0 \\
\lambda & \lambda^{n+\alpha} & -\lambda^{\alpha} & 0
\end{array}\right),
\end{equation*}
where $\lambda=\eta, \eta^{\beta}=\lambda^{\beta}+i\lambda^{n+\beta}$ and $\eta_{\gamma}{}^{\beta}=\lambda_{\gamma}{}^{\beta}+i\lambda_{\gamma}{}^{n+\beta}$.
Then $\eta_{\gamma}{}^{\beta}+\eta_{\bar\beta}{}^{\bar\gamma}=0$ implies that $\zeta$  is a $psh(n)$-valued one form. And (\ref{incon10}) is equivalent to $d\zeta+\zeta\wedge\zeta=0$. Therefore, by calculus on Lie groups, we have that $\zeta$ is the Darboux derivative of some map $f:M\rightarrow PSH(n)$, that is, 
\begin{equation}\label{dadef}
\zeta=f^{*}\omega. 
\end{equation}
Define a map $\phi:M\rightarrow H_{n}$ by $\phi=\pi\circ f$, where $\pi$ is the bundle projection $\pi:PSH(n)\rightarrow H_{n}$, and define a bundle map $\Psi:\xi^{\perp}_M\rightarrow\hat{\xi}^\perp_{1,0}$ by
$\Psi(p,W_{a})=(\phi(p),Z_{a})$. Then, using (\ref{dadef}), it is easy to check that $\Psi$ and $\phi$ satisfy all what we want. This completes the proof.
\end{proof}

\section{Rigidity theorems for submanifolds with CR co-dimensione one}\label{trt}
 In this section, we prove some rigidity theorems for pseudohermitian submanifolds, including both the nondegenerate and  degenerate cases.
\begin{thm}\label{mainthm03}
Let $(M,\hat{J},\hat{\theta})$ be a {\bf vertical}, simply connected pseudohermitian submanifold of $H_{n}$ with CR dimension $m=n-1$. Suppose that the second fundamental form $II=0$.Then $M$ is an open subset $U$ of $H_{n-1}=\{z_{n}=0\}$ after a Heisenberg rigid motion.
\end{thm}
\begin{proof}
In the case $m=n-1$, we write $\theta_{j}{}^{n}=h_{jk}\theta^{k}$, here $h_{jk}$ are the coefficients of the second fundamental form $II$. If $II=0$, then $\theta_{j}{}^{n}=0$. On the other hand, $\nu=0$ implies $\theta^{n}=0$. Hence the structure equations of $H_{n}$, restricting to $M$, reduces to
\begin{equation}\label{ver4}
\begin{split}
d\theta^{j}&=\theta^{k}\wedge\theta_{k}{}^{j},\\
d\theta&=i\theta^{k}\wedge\theta^{\bar k},\\
d\theta_{j}{}^{l}&=\theta_{j}{}^{k}\wedge\theta_{k}{}^{l},\\
d\theta_{n}{}^{n}&=0.
\end{split}
\end{equation}
The last equation of (\ref{ver4}) says that $\theta_{n}{}^{n}$ is closed, and hence locally is exact. By the transformation law of the normal connection, we can choose a normal frame $Z_{n}$ such that the corresponding connection form $\theta_{n}{}^{n}$ vanishes. On the other hand, the first three equation of (\ref{ver4}) is just the structure equations of $H_{n-1}$. This means that $M$ is an open part $U$ of $H_{n-1}\subset H_{n}$, up to a pseudohermitian transformatio $\varphi$ from $M$ to $U$. Define $F$ by $F(x,Z_{n}(x))=(\varphi(x),\mathring{Z}_{n})$. Then $F$ defines the normal bundle isomorphism covering $\varphi$ which preserving the induced pseudohermitian structures, the second fundamental forms and the normal connections of $M$ and $U$, respectively. Hence $\varphi$ is just the restriction of a Heisenberg rigid motion. 
\end{proof}

For a vertical pseudohermitian submanifold of $H_{n}$, we define a flat point of $M$ to be a point such that $II=0$ at that point. Theorem \ref{mainthm06} says that the induced pseudohermitian structure is the only invariant for vertical pseudohermitian submanifolds without flat points.

\begin{thm}\label{mainthm04}
Let $(M,\hat{J},\hat{\theta})$ and $(N,\widetilde{J},\widetilde{\theta})$ be two {\bf vertical}, simply connected pseudohermitian submanifolds of $H_{n}$ without flat points. Suppose both of them are of CR dimension $m=n-1$. If there exists a pseudohermitian transformation $\phi:M\rightarrow N$, then $\phi=\Phi|_{M}$ for some Heisenberg rigid motion $\Phi$.
\end{thm}
\begin{proof}
By theorem \ref{mainthm01}, it suffices to show that both the second fundamental form and the normal connection are completely determined by the induced pseudohermitian structure. We write $\theta_{j}{}^{n}=h_{jk}\theta^{k}$. Then, from the Gauss-like equation, we have
\begin{equation}\label{ver1}
d\theta_{j}{}^{l}-\theta_{j}{}^{k}\wedge\theta_{k}{}^{l}=-h_{jp}h_{\bar{l}\bar{q}}\theta^{p}\wedge\theta^{\bar q}.
\end{equation} 
On the other hand
\begin{equation}\label{ver2}
\begin{split}
d\theta_{j}{}^{l}-\theta_{j}{}^{k}\wedge\theta_{k}{}^{l}&=d\hat{\theta}_{j}{}^{l}-\hat{\theta}_{j}{}^{k}\wedge\hat{\theta}_{k}{}^{l}\\
&=R_{j}{}^{l}{}_{p\bar q}\theta^{p}\wedge\theta^{\bar q}.
\end{split}
\end{equation}
From (\ref{ver1}) and (\ref{ver2}), we see that the Gauss-like equation is equivalent to 
\begin{equation}\label{ver3}
R_{j\bar{l}p\bar q}=-h_{jp}h_{\bar{l}\bar{q}},
\end{equation}
which implies that 
\begin{equation}
R_{j\bar{l}p\bar q}=0 \Leftrightarrow II=0.
\end{equation}
If $II\neq 0$, then there exists $h_{jk}\neq 0$ for some $j,k$. Since $II$ is a symmetric bilinear form, after a frame transformation, we can assume, w.l.o.g., that $h_{11}\neq 0$. Then we have
\begin{equation}\label{ver5}
R_{j\bar{1}p\bar 1}=-h_{jp}h_{\bar{1}\bar{1}}.
\end{equation}
In particular,
\begin{equation}
h_{11}=\sqrt{-R_{1\bar{1}1\bar 1}}e^{i\varphi}, \ \ \ \textrm{for some}\ \varphi.
\end{equation}
On the other hand, if we take another orthonormal frame field $\{\tilde{Z}_{\beta}\}$ such that 
\begin{equation}
\tilde{Z}_{j}=Z_{j},\ \ \tilde{Z}_{n}=e^{i\psi}Z_{n},
\end{equation}
for some $\psi$. Then we have the transformation law for connection forms
\begin{equation}
\begin{split}
\tilde{\theta}_{j}{}^{k}&=\theta_{j}{}^{k}\\
\tilde{\theta}_{j}{}^{n}&=e^{-i\psi}\theta_{j}{}^{n},\ \ \ \ \tilde{\theta}_{n}{}^{j}=e^{i\psi}\theta_{n}{}^{j}\\
\tilde{\theta}_{n}{}^{n}&=\theta_{n}{}^{n}+id\psi.
\end{split}
\end{equation}
Notice that $\theta_{j}{}^{n}=h_{jk}\theta^{k},\ \tilde{\theta}_{j}{}^{n}=\tilde{h}_{jk}\tilde{\theta}^{k}$ and $\tilde{\theta}^{k}=\theta^{k}$, hence we immediately have 
\begin{equation}
\tilde{h}_{jk}=e^{-i\psi}h_{jk},\ \ \ \textrm{for all}\ j,k.
\end{equation}
In particular, $\tilde{h}_{11}=e^{-i\psi}h_{11}=e^{i(\varphi-\psi)}\sqrt{-R_{1\bar{1}1\bar 1}}$. Taking $\psi=\varphi$, we have 
\begin{equation}\label{ver6}
\tilde{h}_{11}=\sqrt{-R_{1\bar{1}1\bar 1}}=\sqrt{-\tilde{R}_{1\bar{1}1\bar 1}}.
\end{equation}
Formula (\ref{ver6}) means that we can always choose a frame field $\{Z_{\beta}\}$ such that $h_{11}=\sqrt{-R_{1\bar{1}1\bar 1}}$, and hence 
\begin{equation}
h_{jk}=-\frac{R_{j\bar{1}k\bar 1}}{\sqrt{-R_{1\bar{1}1\bar 1}}},\ \ \ \textrm{for all}\ j,k.
\end{equation}
This means that the second fundamental form $II$ is completely determined by the induced pseudohermitian structure.

We proceed to show that the normal connection is also completely determined by the induced pseudohermitian structure. For each $j$,
\begin{equation}\label{ver7}
d\theta_{j}{}^{n}=d(h_{jk}\theta^{k})=(dh_{jk}-h_{jl}\theta_{k}{}^{l})\wedge\theta^{k}.
\end{equation}
On the other hand, 
\begin{equation}\label{ver8}
\begin{split}
d\theta_{j}{}^{n}&=\theta_{j}{}^{k}\wedge\theta_{k}{}^{n}+\theta_{j}{}^{n}\wedge\theta_{n}{}^{n}\\
&=(h_{lk}\theta_{j}{}^{l}-h_{jk}\theta_{n}{}^{n})\wedge\theta^{k}.
\end{split}
\end{equation}
From (\ref{ver7}) and (\ref{ver8}), we have, for each $j,k$, 
\begin{equation}
dh_{jk}-h_{jl}\theta_{k}{}^{l}-h_{lk}\theta_{j}{}^{l}+h_{jk}\theta_{n}{}^{n}=\sum_{l=1}^{n-1}B_{jkl}\theta^{l},
\end{equation}
for some $B_{jkl}$, which satisfying $B_{jkl}=B_{jlk}$. In particular
\begin{equation}\label{ver9}
h_{11}\theta_{n}{}^{n}=-(dh_{11}-h_{1l}\theta_{1}{}^{l}-h_{l1}\theta_{1}{}^{l})+\sum_{l=1}^{n-1}B_{11l}\theta^{l}.
\end{equation}
The conjugate of (\ref{ver9}) is,
\begin{equation}\label{ver10}
-h_{11}\theta_{n}{}^{n}=h_{\bar{1}\bar 1}\theta_{\bar n}{}^{\bar n}=-(dh_{\bar{1}\bar 1}-h_{\bar{1}\bar l}\theta_{\bar 1}{}^{\bar l}-h_{\bar{l}\bar 1}\theta_{\bar 1}{}^{\bar l})+\sum_{l=1}^{n-1}B_{\bar{1}\bar{1}\bar l}\theta^{\bar l}.
\end{equation}
Taking the sum of (\ref{ver9}) and (\ref{ver10})
\begin{equation}\label{ver11}
\sum_{l=1}^{n-1}B_{11l}\theta^{l}+\sum_{l=1}^{n-1}B_{\bar{1}\bar{1}\bar l}\theta^{\bar l}=(h_{11,l}\theta^{l}+h_{11,\bar{l}}\theta^{\bar l}+h_{11,0}\theta)+\ \textrm{conjugate},
\end{equation}
which implies that $B_{11l}=h_{11,l}+h_{\bar{1}\bar{1},l}$. Substituting this into (\ref{ver9}), we get
\begin{equation}
\theta_{n}{}^{n}=\frac{h_{\bar{1}\bar{1},l}\theta^{l}-h_{11,\bar l}\theta^{\bar l}-h_{11,0}\theta}{h_{11}},
\end{equation} 
which means that $\theta_{n}{}^{n}$ is completely determined by the induced pseudohermitian structure.
\end{proof}

\begin{rem}
(i) From (\ref{ver11}), we also get $h_{11,0}+h_{\bar{1}\bar 1,0}=0$. Therefore, in the case $n=2$, we have $Th_{11}=0$ or $TR=R_{0}=0$.\\
(ii) Also, for $n=2$, we have that $\theta_{2}{}^{2}=2\theta_{1}{}^{1}-d(\ln{h_{11}})+2Z_{1}(\ln{h_{11}})\theta^{1}$.
\end{rem}

\begin{thm}\label{mainthm05}
Let $(M,\hat{J},\hat{\theta})$ and $(N,\widetilde{J},\widetilde{\theta})$ be two simply connected pseudohermitian submanifolds of $H_{n}$ with CR dimension $m=n-1$. Suppose, in addition, that their fundamental vector fields are {\bf nowhere zero}. If they have the same (induced) pseudohermitian structures. Then they locally differ by a Heisenberg rigid motion. More explicitly, if there exists a pseudohermitian transformation $\phi:M\rightarrow N$, then $\phi=\Phi|_{M}$ for some Heisenberg rigid motion $\Phi$.
\end{thm}

\begin{proof}
The key point is that if $\nu(p)\neq 0$ for each point $p\in M$ , then we can always choose a Darboux frame $p\rightarrow (p;e_{\beta},Je_{\beta},T)$ such that 
\begin{equation}
e_{n}=-\frac{\nu}{|\nu|},\ \ e_{2n}=Je_{n}.
\end{equation}
Then we would like to compute the Darboux derivative of the Barboux frame.  It is equivalent to computing the restrictions of $\theta^{\beta}, \theta_{\gamma}{}^{\beta}$ to $M$. To finish the proof, we need to show the Darboux derivative is completely determined by the induced pseudohermitian structure.
\begin{equation}\label{nver1}
\begin{split}
d\theta^{\beta}&=\theta^{\gamma}\wedge\theta_{\gamma}{}^{\beta}\\
&=\hat{\theta}^{k}\wedge\theta_{k}{}^{\beta}+\theta^{n}\wedge\theta_{n}{}^{\beta}\\
&=\hat{\theta}^{k}\wedge\theta_{k}{}^{\beta}+\hat{\theta}\wedge(-|\nu|\theta_{n}{}^{\beta}).
\end{split}
\end{equation}
On the other hand,
\begin{equation}\label{nver2}
d\theta^{j}=d\hat{\theta}^{j}=\hat{\theta}^{k}\wedge\hat{\theta}_{k}{}^{j}+\hat{\theta}\wedge\hat{\tau}^{j},
\end{equation}
and
\begin{equation}\label{nver3}
d\theta^{n}=-d(|\nu|\hat{\theta})=\hat{\theta}^{k}\wedge(-i|\nu|\hat{\theta}^{\bar k})+\hat{\theta}\wedge(d|\nu|).
\end{equation}
From (\ref{nver1}), (\ref{nver2}) and (\ref{nver3}), there exists complex-valued functions $a^{j}_{\beta\gamma}$ such that $a^{j}_{\beta\gamma}=a^{j}_{\gamma\beta}$ and
\begin{equation}\label{nver4}
\begin{split}
\theta_{k}{}^{j}&=\hat{\theta}_{k}{}^{j}+a^{j}_{kl}\hat{\theta}^{l}+a^{j}_{kn}\hat{\theta},\\
-|\nu|\theta_{n}{}^{j}&=\hat{\tau}^{j}+a^{j}_{nl}\hat{\theta}^{l}+a^{j}_{nn}\hat{\theta},
\end{split}
\end{equation}
Also, there exists complex-valued functions $b_{\beta\gamma}$ such that $b_{\beta\gamma}=b_{\gamma\beta}$ and
\begin{equation}\label{nver5}
\begin{split}
\theta_{k}{}^{n}&=-i|\nu|\hat{\theta}^{\bar k}+b_{kl}\hat{\theta}^{l}+b_{kn}\hat{\theta},\\
-|\nu|\theta_{n}{}^{n}&=d|\nu|+b_{nl}\hat{\theta}^{l}+b_{nn}\hat{\theta}.
\end{split}
\end{equation}
From (\ref{nver4}),
\begin{equation}
\begin{split}
0&=\theta_{k}{}^{j}+\theta_{\bar j}{}^{\bar k}\\
&=(\hat{\theta}_{k}{}^{j}+\hat{\theta}_{\bar j}{}^{\bar k})+a^{j}_{kl}\hat{\theta}^{l}+a^{\bar k}_{\bar{j}\bar l}\hat{\theta}^{\bar l}+(a^{j}_{kn}+a^{\bar k}_{\bar{j}\bar n})\hat{\theta},
\end{split}
\end{equation}
hence 
\begin{equation}\label{nver7}
a^{j}_{kl}=0,\ \ a^{j}_{kn}+a^{\bar k}_{\bar{j}\bar n}=0,\ \ \textrm{for all}\ 1\leq j,k,l\leq m.
\end{equation}
Similarly, and notice that we write $\hat{\tau}^{j}=A^{j}{}_{\bar k}\hat{\theta}^{\bar k}$, we have 
\begin{equation}\label{nver9}
\begin{split}
b_{nl}&=-2(\hat{Z}_{l}|\nu|),\ \ \ b_{nn}+b_{\bar{n}\bar n}=-2(\hat{T}|\nu|);\\
a^{j}_{nl}&=i\delta_{jl}|\nu|^{2},\ \ \ b_{jl}=\frac{A^{\bar j}{}_{l}}{|\nu|},\\
a^{j}_{nn}&=|\nu|b_{\bar{j}\bar n}=-2|\nu|(\hat{Z}_{\bar j}|\nu|),
\end{split}
\end{equation}
for all $1\leq j,l\leq m$.
From (\ref{nver4}), (\ref{nver7}), (\ref{nver9}), we have, fro all $1\leq j,k\leq m$,
\begin{equation}\label{nver10}
\begin{split}
\theta_{k}{}^{j}&=\hat{\theta}_{k}{}^{j}+(i\delta_{jk}|\nu|^{2})\hat{\theta},\\
-|\nu|\theta_{n}{}^{j}&=\hat{\tau}^{j}+(i\delta_{jl}|\nu|^{2})\hat{\theta}^{l}-2|\nu|(\hat{Z}_{\bar j}|\nu|)\hat{\theta}.
\end{split}
\end{equation}
From (\ref{nver5}), (\ref{nver9}), 
\begin{equation}\label{nver11}
\begin{split}
\theta_{k}{}^{n}&=-i|\nu|\hat{\theta}^{\bar k}+\frac{A^{\bar k}{}_{l}}{|\nu|}\hat{\theta}^{l}-2(\hat{Z}_{k}|\nu|)\hat{\theta},\\
-|\nu|\theta_{n}{}^{n}&=(d|\nu|)-2(\hat{Z}_{l}|\nu|)\hat{\theta}^{l}+b_{nn}\hat{\theta}.
\end{split}
\end{equation}
From the look of (\ref{nver10}) and (\ref{nver11}), there is only one term $b_{nn}$ not determined yet. In order to complete the proof, we need to show that both $b_{nn}$ and $|\nu|$ are completely determined by the induced pseudohermitian structure. For this, using (\ref{nver10}) and (\ref{nver11}), we compute
\begin{equation*}
\begin{split}
d\theta_{k}{}^{j}&=\theta_{k}{}^{l}\wedge\theta_{l}{}^{j}+\theta_{k}{}^{n}\wedge\theta_{n}{}^{j}\\
&=\Big(\hat{\theta}_{k}{}^{l}+(i\delta_{lk}|\nu|^{2})\hat{\theta}\Big)\wedge\Big(\hat{\theta}_{l}{}^{j}+(i\delta_{jl}|\nu|^{2})\hat{\theta}\Big)\\
&+\frac{1}{|\nu|^2}\Big(i|\nu|^{2}\hat{\theta}^{\bar k}-A^{\bar k}{}_{l}\hat{\theta}^{l}+2|\nu|(\hat{Z}_{k}|\nu|)\hat{\theta}\Big)\wedge\Big(A^{j}{}_{\bar l}\hat{\theta}^{\bar l}+i|\nu|^{2}\hat{\theta}^{j}-2|\nu|(\hat{Z}_{\bar j}|\nu|)\hat{\theta}\Big),
\end{split}
\end{equation*}
that is,
\begin{equation}\label{nver12}
\begin{split}
&d\theta_{k}{}^{j}-\hat{\theta}_{k}{}^{l}\wedge\hat{\theta}_{l}{}^{j}\\
=&\frac{1}{|\nu|^{2}}\left(-iA^{\bar k}{}_{l}|\nu|^{2}\hat{\theta}^{l}\wedge\hat{\theta}^{j}-iA^{j}{}_{\bar l}|\nu|^{2}\hat{\theta}^{\bar l}\wedge\hat{\theta}^{\bar k}+(-A^{\bar k}{}_{l}A^{j}{}_{\bar q}+\delta_{j}^{l}\delta_{k}^{q}|\nu|^{4})\hat{\theta}^{l}\wedge\hat{\theta}^{\bar q}\right. \\
&\ \ \ \ \  +(2A^{\bar k}{}_{l}|\nu|(\hat{Z}_{\bar j}|\nu|)-2i\delta_{j}^{l}|\nu|^{3}(\hat{Z}_{k}|\nu|))\hat{\theta}^{l}\wedge\hat{\theta}\\&\ \ \ \ \ \ \ \ \left.+(-2A^{j}{}_{\bar q}|\nu|(\hat{Z}_{k}|\nu|)-2i\delta_{k}^{q}|\nu|^{3}(\hat{Z}_{\bar j}|\nu|))\hat{\theta}^{\bar q}\wedge\hat{\theta}\right)
\end{split}
\end{equation}
On the other hand, from (\ref{nver10}) and using the structure equations of the pseudohermitian structure, we have 
\begin{equation}\label{nver13}
\begin{split}
&d\theta_{k}{}^{j}-\hat{\theta}_{k}{}^{l}\wedge\hat{\theta}_{l}{}^{j}\\
=&d\hat{\theta}_{k}{}^{j}-\hat{\theta}_{k}{}^{l}\wedge\hat{\theta}_{l}{}^{j}+d(i\delta_{jk}|\nu|^{2}\hat{\theta})\\
=&R_{k}{}^{j}{}_{p\bar q}\hat{\theta}^{p}\wedge\hat{\theta}^{\bar q}+W_{k}{}^{j}{}_{p}\hat{\theta}^{p}\wedge\hat{\theta}-W^{j}{}_{k\bar p}\hat{\theta}^{\bar p}\wedge\hat{\theta}+i\hat{\theta}_{k}\wedge\hat{\tau}^{j}-\hat{\tau}_{k}\wedge\hat{\theta}^{j}\\
+&i\delta_{jk}\Big((\hat{Z}_{l}|\nu|^{2})\hat{\theta}^{l}\wedge\hat{\theta}+(\hat{Z}_{\bar l}|\nu|^{2})\hat{\theta}^{\bar l}\wedge\hat{\theta}\Big)-\delta_{jk}|\nu|^{2}\hat{\theta}^{l}\wedge\hat{\theta}^{\bar l}.
\end{split}
\end{equation}
Comparing the coefficients of the same terms in (\ref{nver12}) and (\ref{nver13}), and notice that $\hat{\tau}^{j}=A^{j}{}_{\bar k}\hat{\theta}^{\bar k}$, we get
\begin{equation}\label{nver15}
\begin{split}
R_{k}{}^{j}{}_{l\bar q}-\delta_{jk}\delta_{lq}|\nu|^{2}&=-\frac{A^{\bar k}{}_{l}A^{j}{}_{\bar q}}{|\nu|^{2}}+\delta_{j}^{l}\delta_{k}^{q}|\nu|^{2}\\
W_{k}{}^{j}{}_{l}+i\delta_{jk}(\hat{Z}_{l}|\nu|^{2})&=2\frac{A^{\bar k}{}_{l}}{|\nu|}(\hat{Z}_{\bar j}|\nu|)-2i\delta_{j}^{l}|\nu|(\hat{Z}_{k}|\nu|)\\
-W^{j}{}_{k\bar l}+i\delta_{jk}(\hat{Z}_{\bar l}|\nu|^{2})&=-2\frac{A^{j}{}_{\bar l}}{|\nu|}(\hat{Z}_{k}|\nu|)-2i\delta_{k}^{l}|\nu|(\hat{Z}_{\bar j}|\nu|),
\end{split}
\end{equation}
for all $1\leq j,k,l,q\leq m$. From the first equation of (\ref{nver15}), 
\begin{equation}\label{nver26}
\begin{split}
R_{k\bar j}&=R_{k\bar{j}l\bar l}=-\sum_{l=1}^{m}\frac{A^{\bar k}{}_{l}A^{j}{}_{\bar l}}{|\nu|^{2}}+\sum_{l=1}^{m}(\delta_{jk}\delta_{ll}+\delta_{j}^{l}\delta_{k}^{l})|\nu|^{2}\\
&=\left\{\begin{array}{l}-\sum_{l=1}^{m}\frac{A^{\bar k}{}_{l}A^{j}{}_{\bar l}}{|\nu|^{2}},\ \ \textrm{for}\ k\neq j\\
-\left(\sum_{l=1}^{m}\frac{A^{\bar k}{}_{l}A^{j}{}_{\bar l}}{|\nu|^{2}}\right)+(m+1)|\nu|^{2},\ \ \textrm{for}\ k=j.\end{array}\right.
\end{split}
\end{equation}
In particlar
\begin{equation}\label{nver27}
R=R_{k\bar k}=-\frac{|A|^{2}}{|\nu|^{2}}+m(m+1)|\nu|^{2}.
\end{equation}
Formula (\ref{nver27}) is equivalent to 
\begin{equation}\label{nver28}
|\nu|^{2}=\frac{R+\sqrt{R^{2}+4m(m+1)|A|^{2}}}{2m(m+1)}.
\end{equation}
Finally, we would like to compute $b_{nn}$. From (\ref{nver9}), we see that $b_{nn}+b_{\bar{n}\bar n}=-2(\hat{T}|\nu|)$, i.e.,
$b_{nn}=(-\hat{T}|\nu|)+i(\textrm{Im}b_{nn})$. So we only compute $\textrm{Im}b_{nn}$. For this, using (\ref{nver10}) and (\ref{nver11})
\begin{equation}\label{nver31}
\begin{split}
d\theta_{j}{}^{n}&=\theta_{j}{}^{k}\wedge\theta_{k}{}^{n}+\theta_{j}{}^{n}\wedge\theta_{n}{}^{n}=(\theta_{j}{}^{k}-\delta_{j}^{k}\theta_{n}{}^{n})\wedge\theta_{k}{}^{n}\\
&=\frac{1}{|\nu|^{2}}\left(|\nu|\theta_{j}{}^{k}+\delta_{j}^{k}(-|\nu|\theta_{n}{}^{n})\right)\wedge(|\nu|\theta_{k}{}^{n})\\
&=\frac{1}{|\nu|^{2}}\left(|\nu|\hat{\theta}_{j}{}^{k}+\delta_{j}^{k}\big[-(\hat{Z}_{l}|\nu|)\hat{\theta}^{l}+(\hat{Z}_{\bar l}|\nu|)\hat{\theta}^{\bar l}+i(|\nu|^{3}+\textrm{Im}b_{nn})\hat{\theta}\big]\right)\\
&\ \ \ \ \ \ \ \ \ \ \ \ \ \ \ \ \wedge\left(-i|\nu|^{2}\hat{\theta}^{\bar k}+A^{\bar k}{}_{l}\hat{\theta}^{l}-2|\nu|(\hat{Z}_{k}|\nu|)\hat{\theta}\right)\\
&=-i|\nu|\hat{\theta}_{j}{}^{k}\wedge\hat{\theta}^{l}+\hat{\theta}_{j}{}^{k}\wedge\left(\frac{A^{\bar k}{}_{l}}{|\nu|}\hat{\theta}-2(\hat{Z}_{k})\hat{\theta}\right)\\
&+\frac{1}{|\nu|^{2}}\left(-(\hat{Z}_{l}|\nu|)\hat{\theta}^{l}+(\hat{Z}_{\bar l}|\nu|)\hat{\theta}^{\bar l}+i(|\nu|^{3}+\textrm{Im}b_{nn})\hat{\theta}\right)\\
&\ \ \ \ \ \ \ \ \ \ \ \ \ \ \ \ \wedge\left(-i|\nu|^{2}\hat{\theta}^{\bar j}+A^{\bar j}{}_{l}\hat{\theta}^{l}-2|\nu|(\hat{Z}_{j}|\nu|)\hat{\theta}\right)
\end{split}
\end{equation}
On the other hand, using (\ref{nver11}) and the structure equations of the pseudohermitian structure,
\begin{equation}\label{nver32}
\begin{split}
d\theta_{j}{}^{n}&=d\left(i|\nu|\hat{\theta}^{\bar j}+\frac{A^{\bar j}{}_{l}}{|\nu|}\hat{\theta}^{l}-2(\hat{Z}_{j}|\nu|)\hat{\theta}\right)\\
&=-i|\nu|\hat{\theta}_{j}{}^{k}\wedge\hat{\theta}^{\bar k}+i|\nu|\hat{\tau}^{\bar j}\wedge\hat{\theta}-i(d|\nu|)\wedge\hat{\theta}^{\bar j}\\
&-2d(\hat{Z}_{j}|\nu|)\wedge\hat{\theta}-2(\hat{Z}_{j}|\nu|)d\hat{\theta}+d\left(\frac{A^{\bar j}{}_{l}}{|\nu|}\right)\wedge\hat{\theta}^{l}+\left(\frac{A^{\bar j}{}_{l}}{|\nu|}\right)d\hat{\theta}^{l}
\end{split}
\end{equation}
For each $j,\ 1\leq j\leq m$, comparing the coefficients of the terms in (\ref{nver31}) and (\ref{nver32}), we get
\begin{equation*}
\begin{split}
&2\left(\hat{Z}_{\bar j}(\hat{Z}_{j}|\nu|)-\hat{\theta}_{j}{}^{k}(\hat{Z}_{\bar j})(\hat{Z}_{k}|\nu|)\right)-i(\hat{T}|\nu|)\\
=&2\frac{\big|\hat{Z}_{j}|\nu|\big|^{2}}{|\nu|}+|\nu|^{3}+\textrm{Im}b_{nn}-\frac{\sum_{l=1}^{m}|A_{jl}|^{2}}{|\nu|},
\end{split}
\end{equation*}
Writing 
\begin{equation*}
|\nu|_{j\bar j}=\hat{Z}_{\bar j}(\hat{Z}_{j}|\nu|)-\hat{\theta}_{j}{}^{k}(\hat{Z}_{\bar j})(\hat{Z}_{k}|\nu|),
\end{equation*}
and taking the sum for $j$ over $1$ to $m$, we have
\begin{equation}\label{nver33}
-\hat{\Delta}_{b}|\nu|=-\overline{\Box}_{b}|\nu|-im(\hat{T}|\nu|)=2\frac{\big|\hat{\partial}_{b}|\nu|\big|^{2}}{|\nu|}+m(|\nu|^{3}+\textrm{Im}b_{nn})-\frac{|A|^{2}}{|\nu|},
\end{equation}
where $\Box_{b}$ and $\hat{\partial}_{b}$ are the Kohn Laplacian and $\partial_{b}$-operator on $M^{2m+1}$. Hence $b_{nn}$ is determined. Substituting (\ref{nver33}) into (\ref{nver11}), we get
\begin{equation}\label{nver34}
|\nu|\theta_{n}{}^{n}=\hat{\partial}_{b}|\nu|-\bar{\hat{\partial}}_{b}|\nu|+i\left(\frac{\hat{\Delta}_{b}|\nu|}{m}+|\nu|^{3}+\frac{\big(2\big|\hat{\partial}_{b}|\nu|\big|^{2}-|A|^{2}\big)}{m|\nu|}\right)\hat{\theta}.
\end{equation} 
This completes the proof.
\end{proof}

Theorem \ref{mainthm05} says that for pseudohermitian submanifolds of CR dimension $m=n-1$, in which there is no zero for $\nu$, the induced pseudohermitian structure constitute a complete set of invariant. Moreover, we have that if the pseudohermitian torsion of $M$ vanishes, then $M$ locally is part of the standard sphere as Theorem \ref{mainthm05} describes.

\begin{thm}\label{mainthm06}
Let $(M,\hat{J},\hat{\theta})$ be a simply connected pseudohermitian submanifold with CR dimension $m=n-1$. Suppose that the fundamental vector fields is {\bf nowhere zero}. If $A_{\beta\gamma}\equiv 0$, then the Webster curvature $R$ is constant, hence it is part of the standard sphere after a Heisenberg rigid motion.
\end{thm}
\begin{proof}
Suppose that $A_{\beta\gamma}=0$. From (\ref{nver15}) and (\ref{nver27}), we get $R=m(m+1)|\nu|^{2}=constant$. Next we claim
\begin{equation}\label{cla01}
\omega_{a}{}^{2n}=-|\nu|\omega^{a},\ \ \ \textrm{for}\ a=1,\cdots,2n-1.
\end{equation}
From (\ref{nver11}), we have for $1\leq k\leq n-1$,
\begin{equation}\label{cla02}
\theta_{k}{}^{n}=-i|\nu|\hat{\theta}^{\bar{k}}=-i\nu|\theta^{\bar{k}}=-i|\nu|(\omega^{k}-i\omega^{n+k}).
\end{equation}
And from (\ref{nver34}),
\begin{equation}\label{cla03}
\theta_{n}{}^{n}=i|\nu|^{2}\hat{\theta}=-i\nu|\theta^{n}=-i|\nu|(\omega^{n}+i\omega^{2n}).
\end{equation}
On the other hand, we see that for $1\leq k\leq n$, we have 
\begin{equation}\label{cla04}
\theta_{k}{}^{n}=\omega_{k}{}^{n}+i\omega_{k}{}^{2n}=\omega_{n+k}{}^{2n}+i\omega_{k}{}^{2n}.
\end{equation}
Comparing (\ref{cla02}),(\ref{cla03}) and (\ref{cla04}), we get the claim (\ref{cla01}). In addition, we also have 
\begin{equation}\label{cla05}
\omega^{2n}=0,\ \ \omega^{n}=-|\nu|\theta.
\end{equation}
Substituting (\ref{cla01}) and (\ref{cla05}) into the motion equation
\begin{equation}
\begin{split}
de_{2n}&=e_{\beta}\otimes\omega_{2n}{}^{\beta}+e_{n+\beta}\otimes\omega_{2n}{}^{n+\beta}-T\otimes\omega^{n}\\
&=\sum_{\beta=1}^{n-1}e_{\beta}\otimes(|\nu|\omega^{\beta})+e_{n}\otimes(|\nu|\omega^{n})+\sum_{\beta=1}^{n-1}e_{n+\beta}\otimes(|\nu|\omega^{n+\beta})+T\otimes(|\nu|\theta).
\end{split}
\end{equation}
That is
\begin{equation}
d\left(\frac{e_{2n}}{|\nu|}\right)=e_{A}\otimes\omega^{A}+T\otimes\theta=dX,\ \ \ on\ \textrm{M}.
\end{equation}
We conclude that on $M$
\begin{equation}\label{stsp}
X-X^{0}=\frac{e_{2n}}{|\nu|},\ \ \ \textrm{for some }\ X^{0}\in H_{n}.
\end{equation}
Writing 
\begin{equation}
\begin{split}
\frac{e_{2n}}{|\nu|}&=a_{A}(X)\mathring{e}_{A}(X), \ \ \textrm{for some coefficient functions}\ a_{A}\\
X&=(X_{1},\cdots,X_{2n},X_{2n+1})\\
X^{0}&=(X^{0}_{1},\cdots,X^{0}_{2n},X^{0}_{2n+1}).
\end{split}
\end{equation}
From (\ref{stsp}), we have 
\begin{equation}
a_{A}(X)=X_{A}-X_{A}^{0},\ \ \ A=1,\cdots,2n,
\end{equation}
and hence 
\begin{equation}\label{stsp01}
\begin{split}
X-X^{0}&=\frac{e_{2n}}{|\nu|}=a_{A}(X)\mathring{e}_{A}(X)\\
&=(X_{\beta}-X_{\beta}^{0})\left(\frac{\partial}{\partial x_{\beta}}+X_{n+\beta}\frac{\partial}{\partial t}\right)+(X_{n+\beta}-X_{n+\beta}^{0})\left(\frac{\partial}{\partial y_{\beta}}-X_{\beta}\frac{\partial}{\partial t}\right)\\
&=(X_{\beta}-X_{\beta}^{0})\left(\frac{\partial}{\partial x_{\beta}}+X^{0}_{n+\beta}\frac{\partial}{\partial t}\right)+(X_{n+\beta}-X_{n+\beta}^{0})\left(\frac{\partial}{\partial y_{\beta}}-X^{0}_{\beta}\frac{\partial}{\partial t}\right)\\
&=a_{A}(X)\mathring{e}_{A}(X^{0}),
\end{split}
\end{equation}
with $\sum_{A=1}^{2n}a_{A}^{2}=\frac{1}{|\nu|^{2}}$. This completes the proof.
\end{proof}
 
 Theorem \ref{mainthm06} is the same as Theorem \ref{mainthm07}.
\begin{thm}\label{mainthm07}
Let $(M,\hat{J},\hat{\theta})$ be a simply connected pseudohermitian submanifold with CR dimension $m=n-1$. Suppose that the fundamental vector fields is {\bf nowhere zero}. If $II=0$, then the Webster curvature $R$ is constant, hence it is part of the standard sphere after a Heisenberg rigid motion.
\end{thm}
\begin{proof}
Note that in the proof of Theorem \ref{mainthm03}, we choose a Darboux frame such that 
\begin{equation*}
e_{n}=-\frac{\nu}{|\nu|}.
\end{equation*}
This implies that $\big<\nu,Z_{n}\big>=-|\nu|$. Hence, from (\ref{incon2}), we have
\begin{equation*}
\theta_{j}{}^{n}=h_{jk}\hat{\theta}^{k}-i|\nu|\theta^{\bar j},\ \ \textrm{mod}\ \hat{\theta}.
\end{equation*}
Comparing with (\ref{nver11}), we get
\begin{equation*}
h_{jk}=\frac{A_{jk}}{|\nu|},\ \ 1\leq j,k\leq m.
\end{equation*}
Therefore 
\begin{equation*}
A_{jk}=0\Leftrightarrow II=0.
\end{equation*}
This complete the proof.
\end{proof}

\bibliography{main}
\bibliographystyle{plain}

\end{document}